\documentclass[journal,twocolumn,10pt,letter]{IEEEtran}

\IEEEoverridecommandlockouts
\ifCLASSINFOpdf
\else
\fi

\usepackage{bbold}
\usepackage{amsfonts}
\usepackage{nonfloat}
\usepackage{amssymb}
\usepackage[cmex10]{amsmath}
\usepackage{amsthm}
\usepackage{amssymb}
\usepackage{mathrsfs}
\usepackage{amsbsy}
\usepackage{setspace}
\usepackage{graphicx}
\usepackage{newclude}
\usepackage{latexsym}
\usepackage{lettrine} % I added this packagehttps://www.overleaf.com/project/5caf4bdce350a8305ec7a022
\usepackage{tabularx}

\usepackage{amsthm}

\newtheoremstyle{named}{}{}{\itshape}{}{\bfseries}{.}{.5em}{\thmnote{#3's }#1}
\theoremstyle{named}

\usepackage{amsthm}

\theoremstyle{definition}

\newtheorem{definition}{Definition}

\usepackage[table]{xcolor}
\usepackage{multirow}
\usepackage{rotating}
\usepackage{booktabs}
\usepackage{bbm} % indicator function
\usepackage{amsmath,epsfig,cite,amsfonts,psfrag}
\usepackage{lipsum}
\usepackage{cuted}
\usepackage{color}
\usepackage{epstopdf}
\usepackage{float}
\usepackage{mathtools}
\usepackage{bbm}
\usepackage{atbegshi}% http://ctan.org/pkg/atbegshi
\usepackage{balance}
\usepackage[utf8]{inputenc}
\usepackage{multirow}

\usepackage{accents}
\newlength{\dhatheight}

\allowdisplaybreaks

\newtheorem{lemma}{Lemma}

\graphicspath{{figures/}}
\allowdisplaybreaks

\makeatletter
\floatstyle{plain}
\newfloat{twocolequfloat}{b}{zzz}
\floatname{twocolequfloat}{Equation}
\newtheorem{Theorem}{Theorem}
\newtheorem{Remark}{Remark}

\makeatother

\begin{document}

%\title{\LARGE  Art Gallery Problem: A Novel Solution, Analysis, and Application on Indoor Placement of Visible Light Communication Nodes  }

%\title{ Indoor VLC network planning via visibility graph modeling}
%\title{ VLC network planning:the basics of indoor placement via visibility graph modeling}

%\title{ VLC network planning: indoor placement via visibility graph modeling}
%\title{ VLC network planning: indoor placement optimization  via visibility graph modeling}
%\title{ VLC network planning: fundamentals of indoor placement optimization }
%\title{ Visible Light Communication Network: Indoor Planning Optimization }
%\title{ Indoor Planning of VLC Networks to ensure LoS condition in both access and backhauling\vspace{-2mm} }
%\title{ Indoor Network Planning to Ensure LoS Condition for Optical Wireless Access and Backhauling\vspace{-2mm} }
%\title{ Optical Wireless Network Planning to Ensure LoS Condition for Indoor Access and Backhauling\vspace{-2mm} }

\title{ Indoor Planning of Optical Wireless Networks for  LoS Condition in Access and Backhauling\vspace{0mm} }

%\title{Access and Backhauling Planning for Indoor Optical Wireless Networks\vspace{-2mm} }

    \author{\IEEEauthorblockN{Mohsen Abedi, Alexis A. Dowhuszko,~\IEEEmembership{Senior Member,~IEEE}, and Risto Wichman}
    
    \IEEEauthorblockA{Department of Information and Communications Engineering (DICE), \\ Aalto University, 02150 Espoo, Finland\\
    E-mails: \{mohsen.abedi, alexis.dowhuszko,  risto.wichman\}@aalto.fi} \vspace{-10mm}
    }
%\thanks{This work was supported in part by Academy of Finland under Grants 287249, 311752}
\maketitle
\pagestyle{empty}

\begin{abstract}
Optical wireless technology has the potential to complement the wireless access services provided so far over RF. Apart from the abundant unlicensed  bandwidth available for ultra-dense deployments over optical wireless bands, optical wireless also has the potential to offer inexpensive, private, secure, and environmentally friendly communications. However, the main challenge of this technology is the inability to pass through obstacles, requiring a Line-of-Sight~(LoS) condition between transmitter and receiver. In addition, when LEDs are used to provide simultaneously wireless access and illumination, the range of the optical wireless links is notably limited. Since the typical size of Visible Light Communications~(VLC) cells is in the order of few meters, it is challenging to plan the detailed deployment of Access Points~(APs) to prevent coverage holes. This paper proposes a \emph{graph modeling} approach for identifying the minimum number of APs (and their locations) for the given indoor floor plan. A \emph{connectivity tree} is considered to ensure that each VLC AP can communicate with (an)other AP(s) through a LoS infrared wireless link for backhauling. The presented deployment procedure can also control the co-channel interference that is generated throughout the entire indoor environment, enhancing the data rate and illumination performance of VLC networks simultaneously.

\thispagestyle{empty}
\pagestyle{empty}
 %Specially, the new method superiority in service areas with harsh irregularity in demand distributions is highlighted.
\end{abstract}

\vspace{0mm}

%\IEEEpeerreviewmaketitle
\noindent 
%%%%%%%%%%%%%%%%%%%%%%%%%%%%%%%%%%%%%%%%%%%%%%%%%%%%%%%%%%%%%%%%%%%%%%%%%%%%%%%%%%%%%%%%%%%%%%%%%%%%%%%%%%%%%%%%%%%%%%%%%%%%%%%%%%%%

\vspace{-2mm}
\section{Introduction}
\vspace{0mm}

%The Internet of Things (IoT) and machine-type communication networks are increasingly challenging wireless networks today, placing a great deal of strain on their performance, particularly in terms of data rates. Visible Light Communication (VLC) has been determined to be a promising solution in order to improve the data rate and system capacity of the 5G/6G networks that are currently limited by highly saturated RF bands. Approximately $430$ THz - $790$ THz is the range of frequencies used by VLC, a vast and largely untapped area of the electromagnetic spectrum. The ubiquitous use of Light-Emitting Diodes (LEDs), originally deployed for illumination purposes, provides the basis for the widespread adoption of the VLC. Furthermore, VLC systems contribute to sustainability through the combination of illumination and communication by reducing energy consumption and the environmental impact associated with wireless technology. Since the optical signal is confined to a specific area, VLC also provides private and secure communication.

Wireless communications beyond sub-6\,GHz Radio Frequency~(RF) bands, including millimeter waves~(mmWave), Tera-Hertz~(THz), and optical wireless (infrared and visible light) bands~\cite{rapp2019}, are increasingly being considered as a promising option to enable wireless access connectivity beyond 5G. The use of these radio and optical wireless bands will enable a new set of indoor applications involving industrial control processes, connected robotics and autonomous vehicles, mission-critical communications, and connectivity in medical and healthcare applications, among others~\cite{saad2020}. When compared to radio communications on lower frequency bands, mmWave/THz radio and optical wireless technologies offer several advantages, such as wider bandwidths and extremely fast data transfers, the possibility to support a higher density of devices per unit area, enhanced security and reduced interference susceptibility due to the dominance of Line-of-Sight~(LoS) propagation conditions~\cite{zhang2022mmwave,moon20226g, obeed2019optimizing}. Among these frequency bands, Visible Light Communications~(VLC) stands out due to its ability to provide communications as a service on top of the illumination. Apart from the sustainability that emerges thanks to the double use that is given to existing illumination infrastructure, VLC also simplifies notably the co-channel interference confinement, with the potential to enable ultra-densification of indoor deployments~\cite{yadav2018all}.

Despite the potential benefits VLC networks offer, they also experience limitations that impact their ability to provide reliable wireless access everywhere. This is because, in the presence of obstructing (opaque) objects such as walls, doors, and even curtains, the propagation of the VLC signal can be easily blocked between the Light-Emitting Diode~(LED) and the Photodetector~(PD) that takes the role of transmitter and receiver, respectively. Although communication in a VLC system using indirect illumination is possible~\cite{dowh2020a}, the data rate that is feasible is notably affected because the received signal power of a Non-Line-of-Sight~(NLoS) link is much weaker when compared to a LoS link; reason for this is that specular reflections seldom take place on VLC systems, introducing strong attenuation when reflections take place on the objects that are typically found indoors~\cite{matheus2019visible}. Additionally, the optical power that reaches the sensitive area of the VLC receiver falls abruptly to zero when its position falls outside the Field-of-View~(FoV) of the PD~\cite{tabassum2018coverage}, restricting the coverage of VLC cells. 
%This effect, which is representative of optical wireless communication systems and has not been experienced before over RF bands, restricts the maximum coverage range of VLC cells. 
Due to that, detailed VLC network planning is required, which is a straightforward process when the service area of the VLC network extends horizontally without bounds~\cite{borj2020} but becomes more complicated in the presence of walls that block the propagation of VLC signals. 
%As a consequence of this limitation, 
So far, little research has been conducted on the
deployment of VLC Access Points~(APs) to enable seamless indoor wireless access  \cite{abedi2021visible}.

When ultra-densification is applied to provide good quality wireless access in indoor environments, the implementation of the backhaul links to forward the data from/to the core network remains an open challenge regardless of the frequency band that is used. The backhauling technology for indoor wireless communication networks can be either wired or wireless. Fiber optics and copper twisted-pair cables are examples of wired backhauling, with fiber optics being the fastest option \cite{alzenad2018fso,ni2013indoor}. Although the use of a wired backhaul can be considered a reliable way to ensure a wide bandwidth and, as a guided medium, can notably mitigate the interference originating from other co-located systems, it is not cost-effective due to installation and maintenance difficulties. Moreover, it also has limited flexibility to adapt to modifications when optimizing the VLC network after it becomes operative, particularly in the case of ultra-dense deployments. 

 Wireless backhauling offers lower installation and maintenance costs, better flexibility, and faster deployment. Here, RF-based backhauling technologies such as 4G (LTE) or 5G (NR) would be the easiest option to use due to their level of maturity. Still, they would have (relatively) limited bandwidth to offer for this purpose, especially in presence of ultra-dense VLC deployments. Therefore, \emph{optical wireless} backhauling over Infrared~(IR) bands is considered in this paper as an attractive low-cost technology to enable wide bandwidth, high reliability, and low latency communication channel when interconnecting VLC APs (using either highly-directive LEDs or low-power Laser diodes)~\cite{kazemi2018wireless, alzenad2018fso}. Moreover, several publications on VLC networking use Power Line Communications~(PLC) as a sustainable option for backhauling, reusing the star-like topology of the existing electric power transmission cables in buildings for communications. However, apart from the challenges that emerge when transmitting information over a medium that has not been optimized for this purpose, powerline cables behave like a leaky-antenna that irradiates (and receives) a tremendous amount of interference~\cite{yaacoub2021security}.

 Assuming that a VLC AP is not range-constrained for communications, determining the minimum number of these LED-based nodes and their specific locations to provide LoS wireless access within an indoor area confined by walls would be similar to \emph{art gallery problem}, which is known to be NP-hard \cite{de1997computational}. The aim of this problem is to determine the minimum number of camera guards and their locations such that every point of an arbitrary layout is visible by at least one camera guard. Several heuristic solutions  have been proposed to the art gallery problem with a focus on the vertices of the layout, the intersections between edge extensions, and the center points of polygons~\cite{amit2010locating}. Solutions involving the iterative weighting of large sets of grid points~\cite{efrat2006guarding} have been also proposed. However, these heuristic solutions do not consider the requirements of optical wireless communications networks, which include a limited maximum range of a VLC cell and the necessity of a backhaul to connect the VLC APs to the core network. Note that when backhauling is implemented over Optical Wireless~(IR) technologies, the LoS condition between neighboring VLC APs must also be guaranteed. 

% A Photo-Detector PD cannot receive data from an LED when positioned outside its FoV.
% It is this critical two-dimensional distance between a LED and a PD that is referred to as the \emph{maximum range}, denoted by $r$.
% When $r=\infty$, determining the minimum number and locations of LEDs to cover a floor plan is similar to solving the Art Gallery Problem AGP.
% It has been demonstrated that the AGP is NP-hard  \cite{de1997computational}. 
% Based on different candidate sets, such as layout vertices, intersections between edge extensions, and center points of polygons, the authors of \cite{amit2010locating} examine different heuristic solutions for AGP.
% Alternatively, the authors in\cite{efrat2006guarding} propose a solution to AGP by iteratively weighting a selection of points from a large set of grid points.
%   It should be noted, however, that these solutions to AGP are  heuristic and do not admit a variety of VLC network requirements such as maximum range and the optical wireless backhauling.

In this paper, we address the detailed planning of the positions that (minimum number of) VLC APs should take to guarantee a LoS condition for optical wireless access in the indoor service area. With this process, every location inside the floor plan should have visibility to at least one VLC AP. To achieve this task, we first partition the indoor service area and then model it with a graph whose nodes are the subsets of the partition, and the edges indicate the existence of a common visibility area. After that, the \emph{minimum clique cover} method is used to cluster the nodes, minimizing the number of VLC APs and identifying their positions.  %a graph that models the indoor service area has been first proposed and then employed with the \emph{minimum clique cover} method. 
The LoS condition requirement among VLC APs, which enables the configuration of a tree-like network topology for backhauling connectivity, has also been considered. %To this end, the graph is transformed into a tree structure, allowing for direct or clustered visibility between any pair of VLC APs. 
By using our derived lower bounds, it is possible to verify the optimality of the deployment methods under study. Finally, the proposed deployment methods are compared with conventional approaches, highlighting their advantages in terms of the number of VLC APs, achievable data rate, and illumination key performance indicators.
\subsection{Organization}
 Section~\ref{sec:2} presents the system model for the optical wireless network and introduces the detailed planning problem that needs to be tackled. Section~\ref{sec:3} proposes an equivalent graph model for the indoor service areas and, based on this model, Section~\ref{sec:4} derives the deployment strategy of VLC APs to ensure LoS coverage in the wireless access link. In Section~\ref{sec:5}, a tree structure is employed to ensure that the visibility (LoS) condition is observed in the backhaul link that interconnects neighboring VLC APs. Section~\ref{sec:6} presents simulation results and, finally, Section~\ref{sec:7} draws the conclusions.

\subsection{Notations}
 An angle is shown by a caret, and points and polygons on the plane are denoted by capital letters and lower case letters, respectively. Furthermore, $\|.\|$ indicates the Euclidean distance, $\lfloor . \rfloor$ is the floor function, and $|.|$ denotes the size of a set or the amplitude of a complex value.

%The rest of the paper is structured as follows. Section II
%describes the system model. Section III presents the State of the Art in LED\footnote{In this paper, the terms LED and LED are used interchangeably.} deployment and Section IV provides
%proposed methods. Finally, simulation results and conclusions are drawn in Section V and VI, respectively.

\begin{figure}[!t]
\centering
%%%%%%%%%%%%%%%%%figure%%%%%%%%%%%%%%%%%%%%%%
\hspace{0cm}\includegraphics[width=9cm]{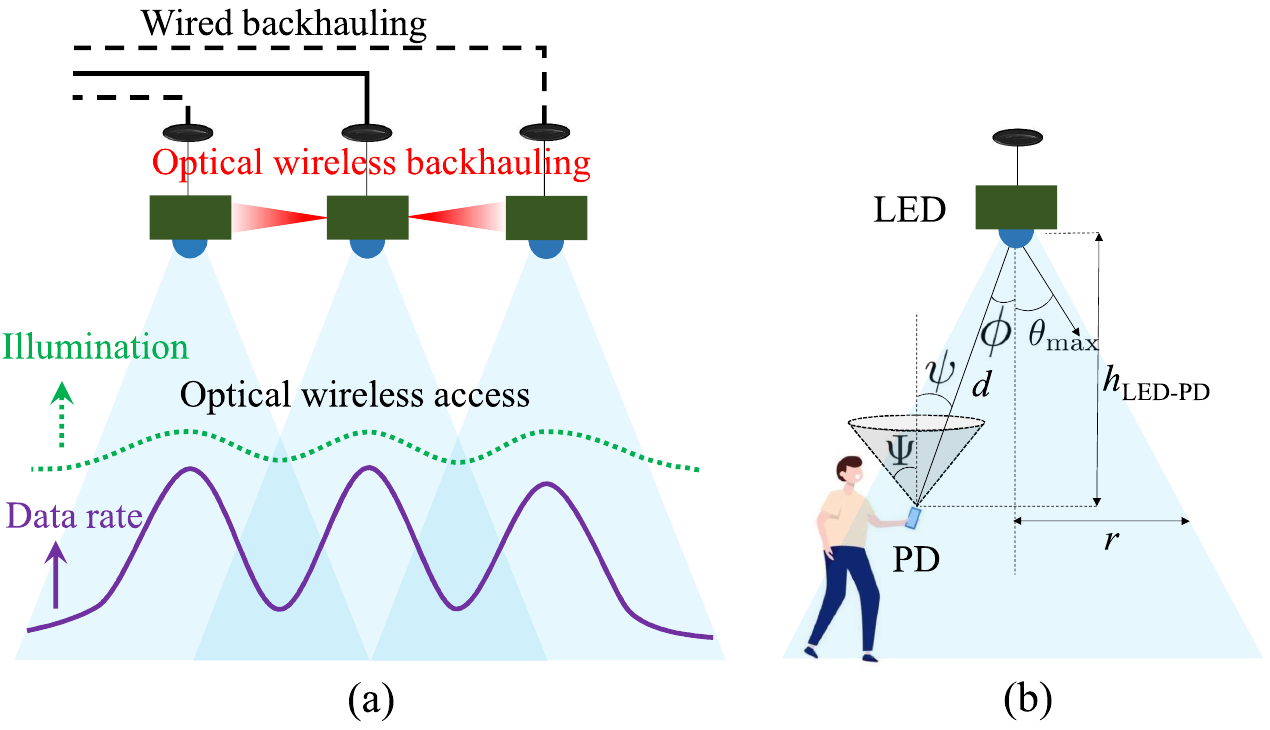}%{vtc_fig0.eps}
    \vspace{-6.8cm}\caption{  Optical wireless communication network overview: a) System model in which VLC APs are interconnected with optical wireless point-to-point infrared links (in red) and provide wireless access over visible light bands (blue). b)  VLC channel model between the visible light LED and the PD.}\label{fig:led}
\vspace{-0.4 cm}
\end{figure}

\color{black}{}

\vspace{-2mm}
\section{System Model and Problem Statement}
\label{sec:2}
%\pagestyle{empty}
%\vspace{0mm}

An optical wireless network is modeled as a set of LED-based APs that transmit data over visible light signals to enable optical wireless access \cite{ghassemlooy2019optical,karunatilaka2015led,boucouvalas2015standards}. Aside from the innate benefits that optical wireless networks offer, such as wireless security and higher bandwidth, such technology has strong potential for sustainability as it  provides data connectivity as a service on top of illumination for indoor environments.
Fig. \ref{fig:led}(a) illustrates the spatial distribution of illumination, which changes more smoothly than the receiving data rate \cite{dowh2017,dowh2020a} in the absence of interference management. In other words, data rates are significantly reduced near cell boundaries and, if this data rate is higher than a target minimum data rate in the whole service area, the cell-range of the VLC system can be mapped into a maximum cell radius $r$, see Fig.~\ref{fig:led}
(b). 

In light of the backhauling technology, Fig. \ref{fig:led}(a) illustrates that a network with wired backhaul links (solid and dashed black lines) is complex and impractical for ultra-dense deployments. 
 As an alternative, optical wireless links with infrared LEDs (or low-power Laser diodes) can be used to connect to nearby VLC APs, which in turn relay this data to other VLC APs until a mesh network is defined for backhaul connectivity. 
Accordingly, reliable optical wireless backhaul requires LoS conditions between VLC APs, either directly or through intermediate nodes.

Indoor VLC AP deployments present significant challenges due to the complexity of the shapes of indoor environments, the presence of obstacles, and the need for reliable, high-speed, and seamless optical wireless access. A number of factors complicate indoor deployment of VLC APs, including the complexity of the floor plan (layout), the range of single VLC cells, and the need for backhauling connectivity that becomes very challenging in ultra-dense deployments.
 With the aid of the VLC channel model, we can determine the maximum range of a VLC AP to achieve a target data rate on the cell-edge areas of the VLC network.   Once we have determined the maximum range, we identify conventional and effective solutions for indoor deployments that provide valuable insights for the detailed planning of a VLC network.
 
\subsection{VLC channel model and optical wireless backhauling}
 To model VLC channels, we use a direct illumination framework as the direct optical signal is much stronger than the reflections \cite{dowh2020a}. 
Here, phosphor-converted LEDs radiate according to Lambertian radiation patterns, as shown in Fig.  \ref{fig:led}(b). Let $K\geq 1$ be the number of LED transmitters LED$-1$, LED$-2$,..., and LED$-K$ that have direct link (visibility) to a PD within the layout. 
 The DC gain of the optical channel between the
LED$-k$ and the PD receiver is given by
\begin{equation}
H_{k,{\rm PD}}^{\rm dir} \hspace{-0.5mm} = \hspace{-0.5mm} \left\{ \hspace{-2mm} \begin{array}{ll}
\frac{(m \hspace{-0.5mm} + \hspace{-0.5mm} 1) A_{\rm PD}}{2 \pi {d_k}^2} \cos^{m} (\phi_k) \cos{(\psi_k)},     &  0 \hspace{-0.5mm} \le \hspace{-0.5mm} \psi_k \hspace{-0.5mm} \le \hspace{-0.5mm} \Psi, \\
   \quad \quad    \quad \quad    \quad \quad 0,  &  \psi_k> \Psi,
\end{array} \right.
\label{eq:2.1}
\end{equation}
where $d_k$\,[m] is the Euclidean distance between LED$-k$ and PD, and $m=-1/\log_{\rm 2}\big[ \cos ( \theta_{\max}) \big]$ denotes the Lambert index of LEDs, in which $\theta_{\max}$\,[rad] defines the source radiation semi-angle at half power. Besides, $\phi_k$\,[rad] and $\psi_k$\,[rad] refer to the angle of irradiance and incidence of the LoS link, respectively, between LED$-k$ and PD. Furthermore, $\Psi$\,[rad] denotes the FoV semi-angle of the PD with an effective physical area of $A_{\rm PD}$\,[m$^2$]. 
Thus, this VLC network has a maximum cell range of $r=h_{\text{LED-PD}} \times \text{tan}(\Psi)$, where $h_{\text{LED-PD}}$ is the vertical distance between LEDs and PD. So, a PD located beyond this range will receive no signal from the LED.

\vspace{0 cm}
\begin{figure}[!t]
\centering
%%%%%%%%%%%%%%%%%figure%%%%%%%%%%%%%%%%%%%%%%
\hspace{0cm}\includegraphics[width=8.6cm]{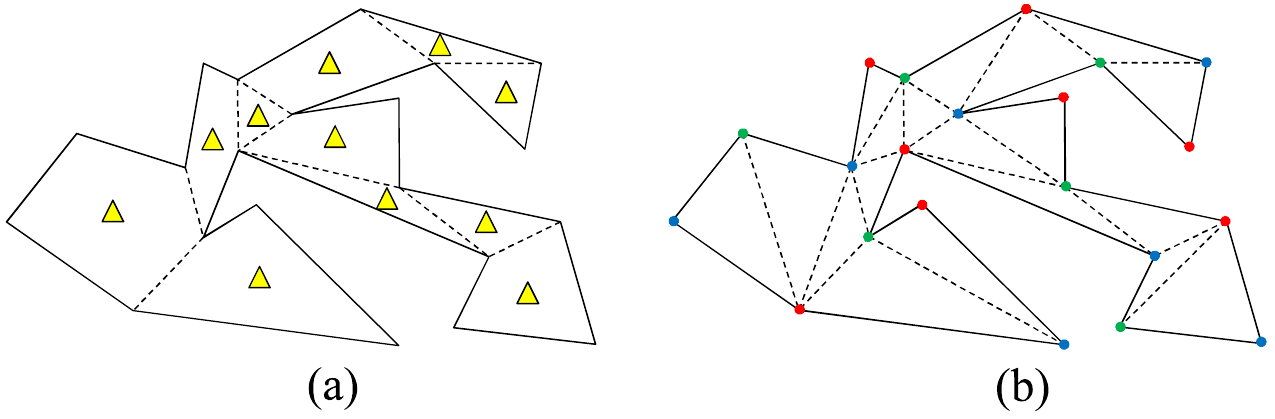}%{vtc_fig0.eps}
    \vspace{-8.7cm}\caption{\ Conventional solutions to the art gallery problem: a) Convex partitioning  by adding non-intersecting diagonals and b) 3-coloring method. }\label{convex-3coloring.eps}
\vspace{-0.6 cm}
\end{figure}

The spectral optical power that reaches the PD from LED$-k$ at wavelength $\lambda$ is given by
\begin{equation}
p_{{ k},{\rm PD}}^{\rm dir}(\lambda) = P_{k} \, H_{k,{\rm PD}}^{\rm dir}  \, S_{\rm o}^{\rm (W)} (\lambda) ,
\label{eq:2.2}    
\end{equation}
where $P_{k}$ and $S_{\rm o}^{\rm (W)}(\lambda)$ are the total radiant power and spectral power distribution of LED$-k$, respectively. Then, the DC current at the output of the PD is given by
\begin{equation}
i_{k,{\rm PD}}^{{\rm dir}} = \int_{\lambda_{\rm l}}^{\lambda_{\rm u}} p^{{\rm dir}}_{k,{\rm PD}}(\lambda) \, \text{R}_{\rm PD}(\lambda) \, f_{\rm o}(\lambda) \, d\lambda,
\label{eq:2.3}    
\end{equation}
%Here, $P_{\rm LED}$ and $S_{\rm o}^{\rm (w)}(\lambda)$ are the total radiant power and spectral power distribution of the LED, respectively.
%Furthermore,
where $\text{R}_{\rm pd}(\lambda)$ denotes the responsivity of the PD and $f_{\rm o}(\lambda)$ is the transmittance of the optical passband filter with lower ($\lambda_{\rm l}$) and upper ($\lambda_{\rm u}$) cutoff wavelengths in the visible light region. 
We then derive the received optical signal strength from LED$-k$ at the PD as
\begin{equation}
%\hspace{0mm} \Gamma_{\rm pd}(f) \hspace{-0.5mm} = \hspace{-0.5mm} \Gamma_{\rm pd}(0) |H_{\rm w}(f)|^2, \hspace{2mm} 
\Lambda_{k,{\rm PD}}  =  \big|  i_{k,{\rm PD}}^{\rm dir} \, G_{\rm tia}  \big|^2,
\label{eq:2.4}    
\end{equation}
where $G_{\rm tia}$ [V/A] is the gain of the Transimpedance Amplifier~(TIA)  embedded into the PD. 

In VLC, the transmitted signals are real and non-negative. Furthermore, the peak power of the transmitters is limited so that capacity-achieving distributions are known to be discrete and to depend on SNR \cite{Smith1971}. Upper and lower bounds for VLC capacity have been derived, e.g., in  \cite{Lapidoth2009} and algorithms to find discrete input distributions have been proposed in \cite{Ma2021}. Here, we simply approximate the data rate received at the PD by %that is feasible to be detected at the output of the PD is given by
\begin{equation}
R_{\text{data}}=B~\text{log}_{\rm 2}\bigg(1+\dfrac{\Lambda_{1,{\rm PD}}}{\sum_{k=2}^K \Lambda_{k,{\rm PD}}+\sigma^2} \bigg),
\label{eq:data-rate}    
\end{equation}
where $B$ and $\sigma^2$ refer to LED modulation bandwidth and optical noise power, respectively. With no loss of generality in this equation, we assumed that the PD receives data from LED$-1$ while the signals from the rest of visible LEDs are considered to be interference. In contrast, the illumination [Lx] per squared meter is given by
 \begin{multline} 
E_{\rm v}= 683 \bigg[ \frac{\rm lm}{\rm W} \bigg]  \int_{\lambda_l}^{\lambda_u}\frac{ \sum_{k=1}^K p_{k,{\rm PD}}^{\rm dir}(\lambda) \, V(\lambda) }{A_{PD}} d\lambda = \\ \frac {\mu_{\rm LED}  \int_{\lambda_l}^{\lambda_u} \sum_{k=1}^K p_{{k},{\rm PD}}^{\rm dir}(\lambda) \,d\lambda }{A_{PD}},
\label{eq:illumination}  
\end{multline}  
 where function $V(\lambda)$ models the human eye sensitivity with tabulated values appear in~\cite{schu2006}, and $\mu_{\rm LED}$ [lm/W] is the luminous efficacy of the white light emitted by the LEDs.
 
% $\lambda_l$ and $\lambda_u$ are the lowest and highest wavelengths in visible light electromagnetic signals, respectively.  Also, the numeric coefficient denotes the average Watt to Lumen conversion ratio. %The illumination at each point is a summation of the illumination factors received from each visible LED.    

%It is possible to approximate the total illumination by averaging the illumination received by each visible LED over the total wavelength per square meter, regardless of the value of $r$.

 %Authors in \cite{dowhuszko2020visible} show that,  as one moves from the VLC cell center to the edges, the spatial distribution of the illumination changes more smoothly than that of the receiving data rate, see Fig. \ref{fig:led}b).  As is shown in Fig. \ref{fig:led}b), wired backhauling is complex and infeasible for ultra-dense networks. However in OWB, one LED transfers backhauling data to a visible LED nearside via infra-red which in turn it relays the this data partly to to others until all LEDs are covered. 

%\vspace{0mm}
%\section{State of the art in indoor vlc deployment}
%\label{sec:2}
%\pagestyle{empty}
%\vspace{0mm}

\vspace{-2mm}
\subsection{Indoor deployment of LEDs with unconstrained range for LoS coverage}
\label{sec:2b}
\vspace{0mm}

The indoor deployment of LED-based APs for LoS coverage presents a challenge similar to the well-known art gallery problem, particularly when assuming an unconstrained range ($r=\infty$) for the VLC cells. A single camera guard (or VLC AP) with unlimited range can cover any point within a convex layout. However, when dealing with non-convex layouts, such as the one with $n=21$ vertices depicted in Fig.~\ref{convex-3coloring.eps}, multiple guards are required for LoS coverage~\cite{paalsson2008camera}. 

 %The problem of LED deployment for LoS coverage becomes similar to the renowned art gallery problem when assuming unconstrained range $r=\infty$ for LEDs. For example, a single camera guard (LED with $r=\infty$) covers any point inside a convex layout, but the sample layout with $n=21$ vertices depicted in Figure \ref{convex-3coloring.eps} requires multiple guards for LoS coverage .  

\vspace{0 cm}
\begin{figure}[!t]
\centering
%%%%%%%%%%%%%%%%%figure%%%%%%%%%%%%%%%%%%%%%%
\hspace{0cm}\includegraphics[width=8.7cm]{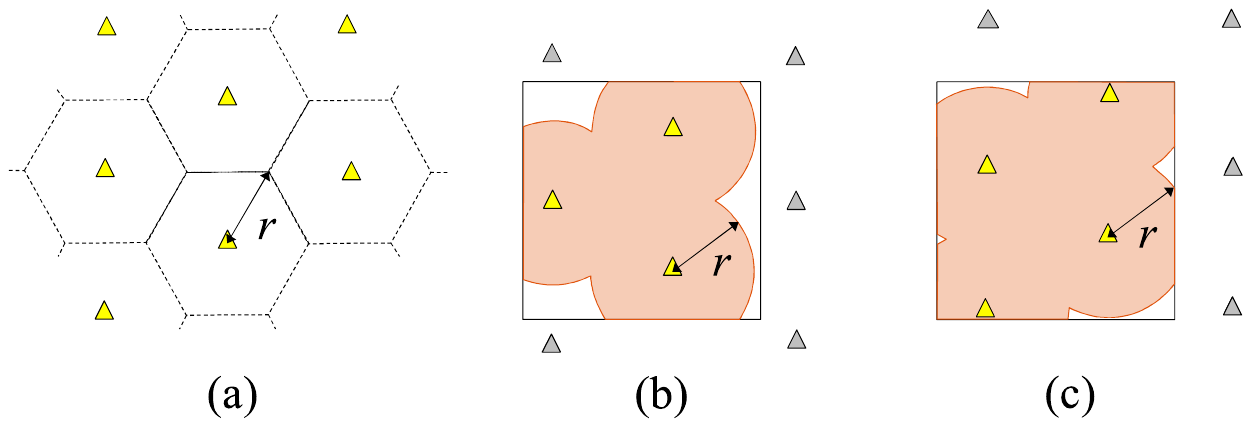}%{vtc_fig0.eps}
    \vspace{-8.7cm}\caption{ Conventional deployment of VLC APs with limited range: 
    a) Hexagonal cells in a boundless area without obstacles;
    b) Square indoor area  capturing three VLC cells;  
    c) Square indoor area with four VLC cells. } \label{fig:squar-hexagonal}
\vspace{-0.2 cm}
\end{figure}

%Assuming the unconstrained range $r=\infty$, the problem of LED deployment for indoor coverage resembles the well-known art gallery problem. Literature on the AGP attempts to identify the minimum number and placement of omnidirectional camera guards required to cover all points in an arbitrary layout.
% For instance, one guard can be placed anywhere inside a convex layout, whereas multiple guards are required to cover the sample layout  $\mathbb{L}$ with $n=21$ vertices in Fig. \ref{convex-3coloring.eps}b), see \cite{dobbins2018smoothed}.

The conventional solutions for the art gallery problem involve Convex Partitioning~(CP) methods and 3-coloring. CP methods aim at reducing a given layout into the fewest number of convex components.  Fig. \ref{convex-3coloring.eps}(a) depicts a CP method that adds the smallest set of non-intersecting diagonals~\cite{chazelle1994decomposition}. Here, placing one camera guard  on the center point of each convex component results in $11$ guards to cover the whole sample layout, respectively. Chvatal suggested that the upper bound for the minimum number of guards required is $\big\lfloor n/3 \big\rfloor$, which Fisk then verified by demonstrating that the vertices of every layout are 3-colorable~\cite{de1997computational}. With this proof, each vertex is labeled with one of three different colors such that no adjacent vertices share the same color. Thus, placing guards on vertices with the minimum number of color uses, such as six green vertices in  Fig. \ref{convex-3coloring.eps}(b), provides full coverage of the layout. However, the art gallery problem does not address the unique requirements of optical wireless networks, which are limited range and the requirement of optical wireless backhauling.

%Conventional solutions to AGP include  Convex Partitioning CP methods and 3-coloring. CP methods  reduce a given layout to the fewest number of convex components.
% Fig. \ref{convex-3coloring.eps}a) shows a CP method by adding the smallest set of non-intersecting diagonals, while Fig. \ref{convex-3coloring.eps}b) represents a CP  by introducing segments that bisect the reflex vertices of the given layout \cite{hertel1985fast,chazelle1994decomposition}. 
% Then, the placement of one camera guard (LED with $r=\infty$) on the center point of each convex component results in $11$ and $9$ guards, respectively, that fully cover the sample layout.
% Chvatal suggested that $\big\lfloor n/3 \big\rfloor$ marks the upper bound for the minimum number of guards, which Fisk verified by showing that the vertices of every layout are 3-colorable, \cite{fisk1978short,de1997computational}.
%  With this proof, each vertex is labeled with one of three different colors, so that no adjacent vertices are the same color. So, placing guards on the vertices with the minimum number of color uses, e.g., $6$ green vertices, covers the whole layout, see Fig. \ref{convex-3coloring.eps}c).

\vspace{0 cm}
\begin{figure}[!t]
\centering
%%%%%%%%%%%%%%%%%figure%%%%%%%%%%%%%%%%%%%%%%
\hspace{0cm}\includegraphics[width=9cm]{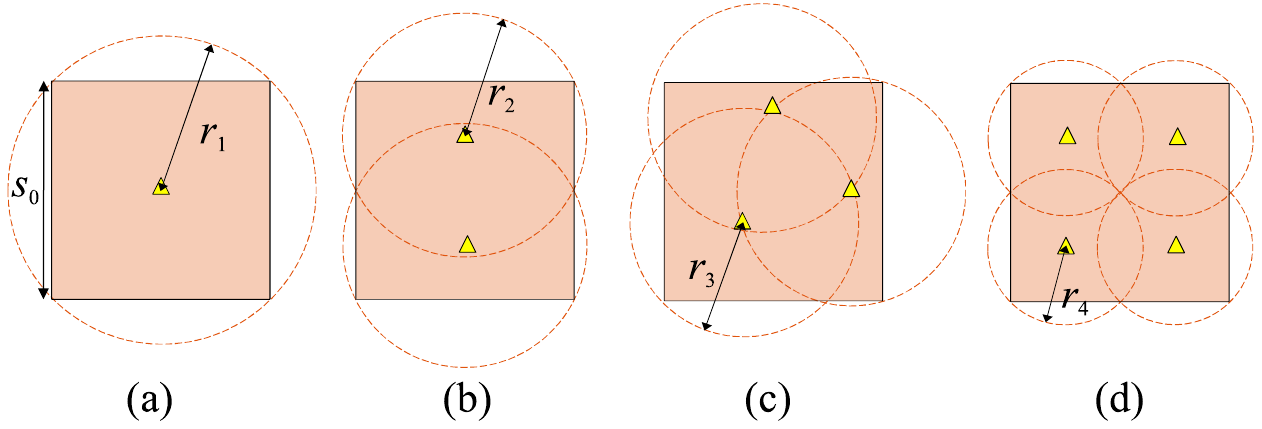}%{vtc_fig0.eps}
    \vspace{-9cm}\caption{ LoS coverage of the square-shaped indoor area with side length $s_0$ by optimal deployment of: a) A single VLC AP with coverage range $r_1=\sqrt{2}s_0/2$; b) two VLC APs with coverage range  $r_2=\sqrt{5}s_0/4$; c) three VLC APs with coverage range $r_3= (\sqrt{6}-\sqrt{2})s_0/2$; and d) four VLC APs with coverage range $r_4=\sqrt{2}s_0/4$.} \label{fig:squar-optimal}
\vspace{-0.5 cm}
\end{figure}

\vspace{-2mm}
\subsection{Deployment of LEDs with limited range for LoS coverage}
\label{sec:2bb}
\vspace{0mm}
In a boundless area without obstacles, hexagonal  cells are the most efficient architecture for deploying VLC APs with limited range, just like RF-based mobile networks deploy macrocells to provide coverage outdoors. As shown in  Fig. \ref{fig:squar-hexagonal}(a), hexagonal cells  ensure LoS coverage by deploying the minimum number of VLC APs. However, when it comes to a limited-sized area with walls, as in Fig. \ref{fig:squar-hexagonal}(b), deploying LEDs becomes particularly challenging.
\emph{Hex deployment} refers to the method of shifting the square room over the hexagonal cells to maximize the coverage area, as shown in Fig. \ref{fig:squar-hexagonal}(c), which may still result in outage areas.

%To overcome this challenge, several techniques can be employed, such as adjusting the intensity and direction of the LEDs, or using reflectors or diffusers to enhance the coverage. Additionally, simulations and testing can be carried out to optimize the placement and configuration of the LEDs for a specific area.

%In a boundless (infinitely large) area without obstacles, hexagonal cells are the preferred architecture for LEDs with limited range  in the same way as for RF cellular networks outdoors. As illustrated  in  Fig. \ref{fig:squar-hexagonal}a), hexagonal cells conventionally ensure seamless coverage by deploying a minimum number of LEDs.
%Despite the boundless areas, deployment of LEDs in a limited-sized area with walls\footnote{ Hereafter, we use only the term indoor area to be concise.} is particularly challenging. A solution can be shifting the set of LEDs horizontally and vertically to maximize the coverage area of the square room as illustrated in  Fig. \ref{fig:squar-hexagonal}b) and c). However, a seamless coverage is not achievable in this case. 

For the square room with side length $s_0$ as in Fig.~\ref{fig:squar-optimal}(a), deploying a single VLC AP in the center with maximum range $r=\sqrt{2}s_0/2$ or larger is enough for a LoS coverage. Furthermore, as in Fig.~\mbox{ \ref{fig:squar-optimal}(b)}, two VLC APs with a maximum range  $\sqrt{5}s_0/4\leq r<\sqrt{2}s_0/2$ are enough to cover the area. Interestingly, Fig. \ref{fig:squar-optimal}(c) illustrates that when range verifies $(\sqrt{6}-\sqrt{2})s_0/2 \leq r < \sqrt{5}s_0/4 $, three VLC APs are necessary and sufficient for a LoS coverage. Lastly, four VLC APs are needed when range is limited to $\sqrt{2}s_0/4 \leq r< (\sqrt{6}-\sqrt{2})s_0/2$. 
As a result, since it is not straightforward to deploy VLC APs to provide LoS coverage in a square room, it is expected that irregular layouts will further complicate this task. Furthermore, when considering the optical wireless backhauling, the optimal VLC AP deployment requires more sophisticated algorithms. So, we begin by developing a graph model for indoor areas. With the aid of this model, we can determine the optimal placement of VLC APs to ensure a reliable optical wireless access. 

%It is observed that, deploying LEDs for a seamless coverage of even a square-shaped room is not straightforward. When it comes to irregular layout, the optimal deployment would be more challenging. Moreover, taking into account the OWB, optimal deployment of LEDs require more sophisticated algorithms. Therefore, we firstly discuss the deployment method for a seamless coverage. Then, we add connectivity requirement to ensure a reliable OWB. 

%Given a  square-shaped room with side length $s=\sqrt{2}r$ and smaller, a single LED in the center provides full coverage.
% Furthermore, a square with $\sqrt{2}r <s\leq0.8\sqrt{5}r$ is covered by two LEDs. For the side length interval $0.8\sqrt{5}r <s\leq 2r$, three LEDs are necessary and sufficient for full coverage.
% Lastly, four LEDs are required to cover the square area of side length $2r <s\leq 2\sqrt{2}r$.
%  It shows that deploying LEDs to cover even a square indoor area is not straightforward.

%\vspace{-1mm}
%\subsection{unconstrained art gallery problem}
%\label{sec:2a}
%\vspace{-0.5mm}

%\vspace{0mm}
%\section{The range constrained art gallery problem}
%\label{sec:3}
%\vspace{0mm}

\vspace{-2mm}
\section{Equivalent Graph Model for  an Indoor Area}
\label{sec:3}
\vspace{0mm}
In order to construct an equivalent optimization problem, we characterize a visibility graph as a tool for modeling the geometrical characteristics of a service area layout.

 Given a point $X$, a polygon $p$, and an area $\mathcal{A}$, we first present underlying definitions.

\begin{definition}\label{point_visibility.def}
  $\bf{Visibility~area~of~a~ 
 ~point}$ $\mathcal{V}(X)$ denotes the set of all the points $Y$ such that:
 \begin{itemize}
     \item The line segment $XY$ lies entirely inside the layout, and
     \item $\|XY\|\leq r$.
 \end{itemize}\end{definition}So, there is no edge between $X$ and $\mathcal{V}(X)$, and $\mathcal{V}(X)$ is within the maximum range from $X$.
 
   \begin{definition}\label{polygon_visibility.def}
 $\bf{Visibility~area~of~a~
 ~polygon}$ $\mathcal{V}(p)$ refers to the set of all the points $Y$, such that $Y\in\mathcal{V}(X_i)$ for $i=1, 2,..., l$, where $X_1, X_2, ..., X_l$ denote the vertices of $p$.
\end{definition}
From Definitions \ref{point_visibility.def} and \ref{polygon_visibility.def}, it follows immediately that  $\mathcal{V}(p)=\mathcal{V}(X_1)\cap\mathcal{V}(X_2)... \cap \mathcal{V}(X_l)$. There is, however, still a need for area-to-area mapping in indoor environments.

 \begin{definition}\label{Connectivity_region.def}
 $\bf{Connection~Region~of~an 
 ~area}$ $\mathcal{CR}(\mathcal{A})$  corresponds to the locus $Y$ for which there is at least one point $X\in\mathcal{A}$ such that the segment $XY$ entirely lies inside the layout.
\end{definition}

Given the point $X$ and the polygon $p$, Fig.~\ref{fig-visibility.pdf}(a) visualizes $\mathcal{V}(X)$  and $\mathcal{V}(p)$ in red and yellow, respectively, with the maximum range $r$. It shows that 
all vertices of $p$ are openly visible and lie within $r$ of any point in $\mathcal{V}(p)$. Further, given the area $\mathcal{A}$, Fig.~\ref{fig-visibility.pdf}(b) indicates $\mathcal{CR}(\mathcal{A})$ in green, whose points are openly visible by at least one point in $\mathcal{A}$.

  Partitioning the layout is the first step in creating an equivalent graph model.  There is a significant impact of partitioning on the complexity of the graph, as well as on optimization problems that result from it.  In terms of layout parts, triangles are particularly attractive among all types of polygons due to their convexity and ability to partition any layout.
%\footnote{The triangulation process is performed in polynomial time. In this situation, the visibility area of a point is computed in linear time, while deriving the intersection of two areas has the complexity of $\mathcal{O}(n\log n)$, see \cite{de1997computational,el1981linear,margalit1989algorithm}.}
Triangulation  refers to the process of partitioning a layout into triangles by adding diagonals. However, to meet the requirements of optical wireless networks, these triangles should be smaller than in the triangulation process.

 \begin{definition}\label{hyper-triangulation.def} $\bf{Hyper~Triangulation~method}$  $\mathcal{HT}(R)$ begins with the triangulation process and continues by connecting the midpoint of the largest side at each triangle to its opposite vertex, until no triangle with a side that is larger than a target parameter $R$ is left.
\end{definition}

\vspace{0 cm}
\begin{figure}[!t]
\centering
\advance\leftskip-0.7cm
\advance\rightskip-1cm
%%%%%%%%%%%%%%%%%figure%%%%%%%%%%%%%%%%%%%%%%
\hspace{-0.3cm}\includegraphics[width=9.2cm]{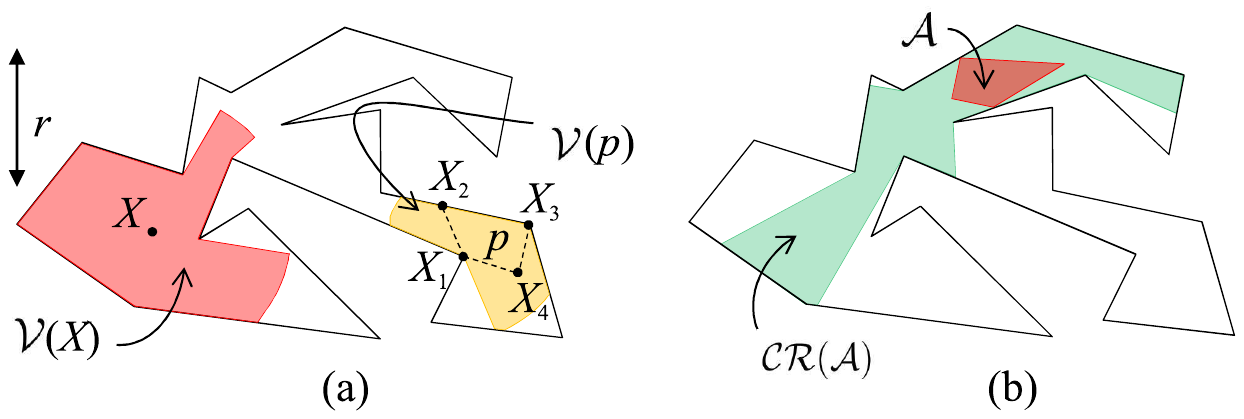}%{vtc_fig0.eps}
\vspace{-9.1cm}\caption{  a) Visibility area of a point and a polygon with the maximum range $r$. b) Connection region of an area. }\label{fig-visibility.pdf}
\vspace{-0.1 cm}
\end{figure}

%\\
%\vspace{1 cm}

\begin{figure}[!t]
\centering
\advance\leftskip-0.8cm
\advance\rightskip-1cm
%%%%%%%%%%%%%%%%%figure%%%%%%%%%%%%%%%%%%%%%%
\includegraphics[width=9.2cm]{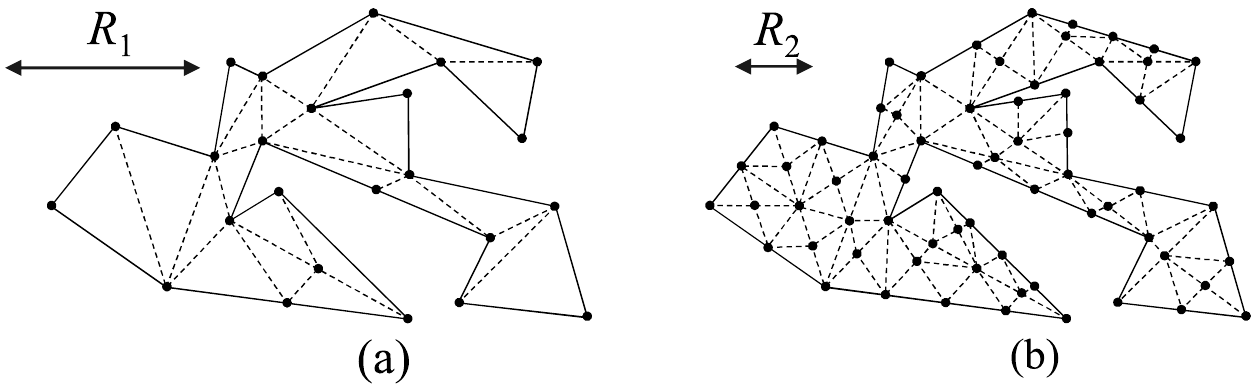}%{vtc_fig0.eps}
\vspace{-9.3cm}\caption{  Hyper triangulation with two different parameter values: a) $\mathcal{HT}(R=R_1)$ and b) $\mathcal{HT}(R=R_2)$.}\label{fig-hyper-triangulation}
\vspace{-0.4 cm}
\end{figure}

Figure \ref{fig-hyper-triangulation} exhibits $\mathcal{HT}(R=R_1)$ and $\mathcal{HT}(R=R_2)$, where $R_1>R_2$. Starting from the triangulation  $\mathcal{HT}(R=\infty)$ as in Fig.~\ref{convex-3coloring.eps}(c), Fig.~\ref{fig-hyper-triangulation}(a) is derived by bisecting three triangles whose largest sides are longer than $R_1$. Similarly, in Fig.~\ref{fig-hyper-triangulation}(b), it is verified that all the sides of the triangles are smaller than $R_2$. As $R$ decreases in $\mathcal{HT}(R)$, the total number of triangles $M$ increases with an approximate rate of $1/R^2$ for small values of $R$.
The smaller $r$ requires the smaller $R$ to be set in hyper triangulation to ensure the polygons have sufficient visibility areas.
Here, we  characterize the properties of the hyper triangulation through Lemma \ref{lemma:non-empty-visibility} to  \ref{Lemma:visibility_to_inside_triangle}.

\begin{lemma}\label{lemma:non-empty-visibility}
Let $p$ be a triangle in a hyper triangulation from the space $\mathcal{HT}(R\leq \sqrt{3}r)$. Then, $\mathcal{V}(p)\neq \emptyset$.
\end{lemma}
\begin{proof} See Appendix.
 \end{proof}
 
According to Lemma \ref{lemma:non-empty-visibility}, setting $R$ to any value less than $ \sqrt{3}r$ ensures a non-empty visibility area for all the triangles in the hyper triangulation.

\begin{lemma}\label{lemma:Covering-visibility}
Let $p$ be a triangle in a hyper triangulation from $\mathcal{HT}(R\leq r)$. Then, $Q\in\mathcal{V}(p)$ for any point $Q$ inside $p$.
\end{lemma}

\begin{proof} See Appendix.
\end{proof}From Lemma \ref{lemma:Covering-visibility}, setting $R\leq r$ guarantees that the visibility area of each triangle at least covers itself.

\begin{lemma}\label{Lemma:visibility_to_inside_triangle}
Let $p$ be a triangle in $\mathcal{HT}(R)$. If there is a point $Y$ such that $Y\in \mathcal{V}(p)$, then $Y\in \mathcal{V}(Q)$, wherein $Q$ refers to any point inside $p$.
\end{lemma} 
\begin{proof} See Appendix.
\end{proof}

Using definitions \ref{point_visibility.def} and \ref{polygon_visibility.def}, we define the graph model: 
%Here, we define the graph model.   

\begin{definition}\label{definition:PV-graph}$\hspace{-0.1cm}\bf{Partition\hspace{-0.1cm}-\hspace{-0.1cm}based~Visibility~(PV)~ graph}$ refers to a simple unweighted  graph whose nodes represent the triangles (polygons) $p_1$, ...., $p_M$ of the layout. Two nodes $p_i$ and $p_j$ are adjacent if and only if $\mathcal{V}(p_i)\cap \mathcal{V}(p_j)\neq \emptyset$ for $1\leq i,j\leq M$ and $i\neq j$.
\end{definition}

The structure of the PV graph depends on the shape of the layout, the value of $R$ for $\mathcal{HT}(R)$, and the cell range $r$.
 PV graphs provide valuable information, such as an infinite set of possible VLC AP placements for LoS coverage; it is also an effective tool for assessing the optimality of a proposed VLC AP deployment. Further, it provides a basis for considering a variety of requirements in optical wireless networks.

\begin{definition}\label{Clique_visibility.def}
 $\bf{Visibility~Area~of~a 
 ~Clique}$ $\mathcal{V}(c)$ refers to the intersection of the visibility areas of all the nodes in a clique $c$ from the PV graph, i.e., 
 $\mathcal{V}(c)=\bigcap_{p_i\in c}\mathcal{V}(p_{i})$.
\end{definition}
 %We can also express Definition \ref{Clique_visibility.def} in an abstract form as $\mathcal{V}(c)=\bigcap_{p_i\in c}\mathcal{V}(p_{i})$.
     %Based on graph theory, cliques are considered complete subsets of a graph.

\vspace{-3mm}
\section{Deployment of VLC APs to ensure LoS in wireless access links}
\label{sec:4}
\vspace{0mm}

For LoS coverage in an indoor area, we aim to determine the minimum number of VLC APs and their locations. For this, it is necessary to cover every point in a layout with at least one VLC AP to achieve LoS coverage.
Hence, in this section, we consider the use of clique partitioning as a feasible deployment method of VLC APs to enable LED-based wireless access.

  %\vspace{1 cm}
\begin{figure}[!t]
\centering
%%%%%%%%%%%%%%%%%figure%%%%%%%%%%%%%%%%%%%%%%
\hspace{0cm}\includegraphics[width=8.3cm]{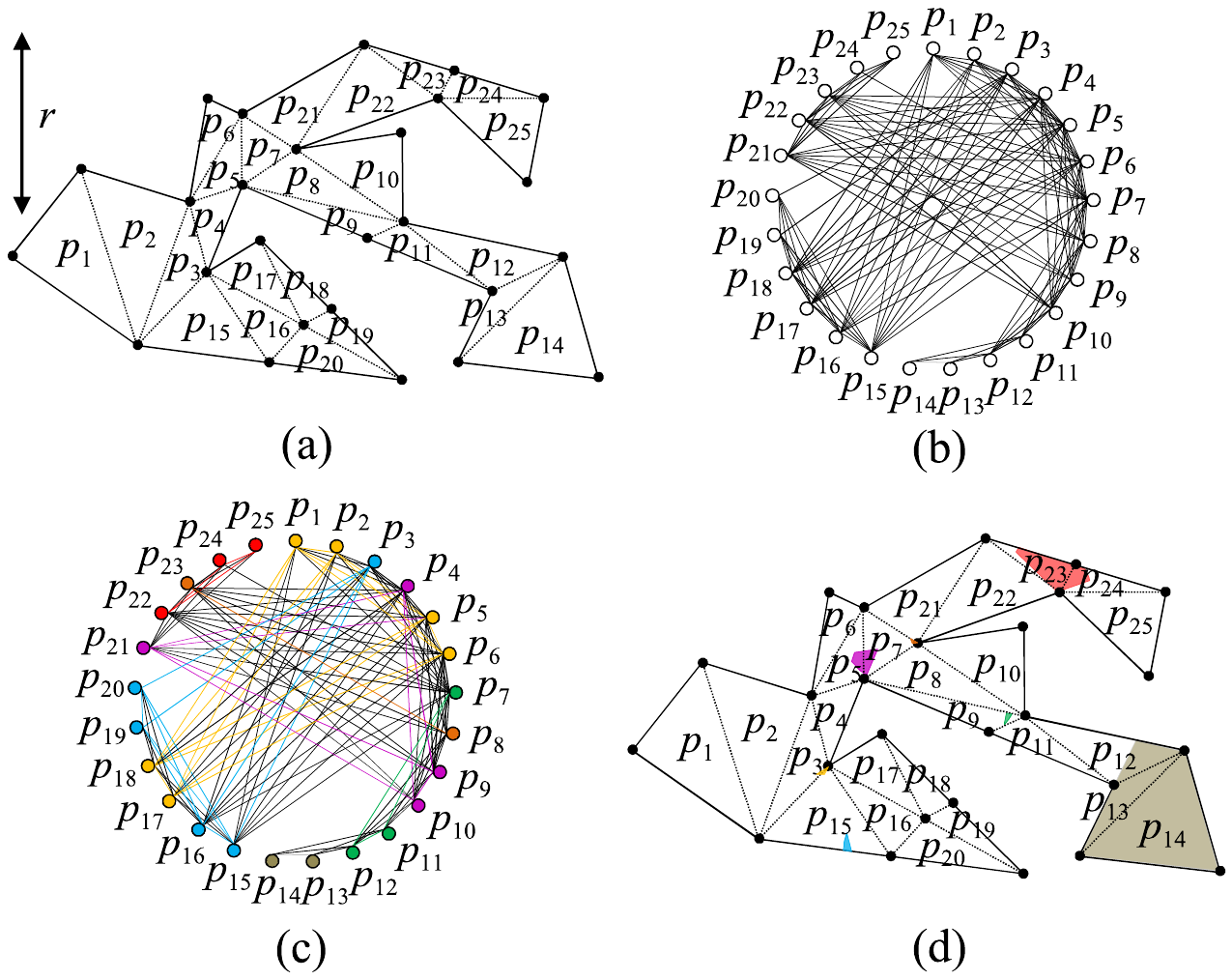}%{vtc_fig0.eps}
\vspace{-4.5cm}\caption{  A feasible VLC AP deployment to meet LoS coverage via a non-optimal clique partitioning of the PV graph: a) The sample layout partitioned by $\mathcal{HT}(R=r)$; b) the resulting PV graph of the partitioned layout; c) a non-optimal clique clustering of the PV graph; d) visibility areas of the cliques.}  \label{fig-nonoptimal1.pdf} %a) The hyper triangulation $\mathcal{HT}(R=r)$ for $\mathbb{L}$. b) The PV graph  constructed by $\mathcal{HT}(R=r)$.  c) A feasible clique partitioning. d) Respective placement areas of the cliques.} 

\vspace{-0.4 cm}
\end{figure}
%\vspace{0mm}
%\subsection{Indoor coverage (globecom)}
%\label{sec:4}
%\vspace{0mm}

\begin{Theorem}\label{Theorem:clique_partitioning}
Assume partitioning a PV graph into $g$ cliques $c_1$, $c_2$,..., $c_g$, such that $\mathcal{V}(c_j)\neq \emptyset$ for  all $j$. Then, deploying a set of $g$ VLC APs, one anywhere inside each $\mathcal{V}(c_j)$ ensures LoS coverage in access.
\end{Theorem}
\begin{proof}
 From Lemma \ref{lemma:non-empty-visibility}, the space $\mathcal{HT}(R\leq \sqrt{3}r)$ guarantees that no node in the PV graph has an empty visibility area. Besides,
 Lemma \ref{Lemma:visibility_to_inside_triangle} implies that a VLC AP anywhere inside $\mathcal{V}(c_j)\neq \emptyset$ covers every point inside the triangles $p_{i}\in c_j$. 
 Finally, all nodes in the PV graph are covered by the union of all cliques $\bigcup_{j=1}^g c_j$, thus confirming Theorem \ref{Theorem:clique_partitioning}.
\end{proof}

Fig. \ref{fig-nonoptimal1.pdf} illustrates a feasible deployment of VLC APs to provide LoS coverage in the sample layout with the maximum range $r$ using a non-optimal clique partitioning of the PV graph. Fig.~\ref{fig-nonoptimal1.pdf}(a) shows $\mathcal{HT}(R=r)$ with $M=25$ triangles, for which Fig.~\ref{fig-nonoptimal1.pdf}(b) represents the PV graph. For instance,  since $\mathcal{V}(p_3)\cap\mathcal{V}(p_{20})\neq \emptyset$ and $\mathcal{V}(p_3)\cap\mathcal{V}(p_{22})= \emptyset$, the nodes $p_{3}$ and $p_{20}$ are adjacent, while $p_{3}$ and $p_{22}$ are non-adjacent. Fig.~\ref{fig-nonoptimal1.pdf}(c) displays the partition of the PV graph into seven cliques distinguished by seven different colors, whose visibility areas are illustrated in Fig.~\ref{fig-nonoptimal1.pdf}(d) with similar colors. For example, the purple area  in the layout represents the visibility area of the clique with the nodes $p_4$, $p_9$, $p_{10}$, and $p_{21}$. Hence, deploying $g=7$ VLC APs anywhere inside each clique visibility area, as in Fig.~\ref{fig-nonoptimal1.pdf}(d), guarantees LoS coverage for the sample layout.

     %\vspace{1 cm}
\begin{figure*}[!t]
\centering
%\captionsetup{justification=centering}
%%%%%%%%%%%%%%%%%figure%%%%%%%%%%%%%%%%%%%%%%
\hspace{0cm}\includegraphics[width=14cm]{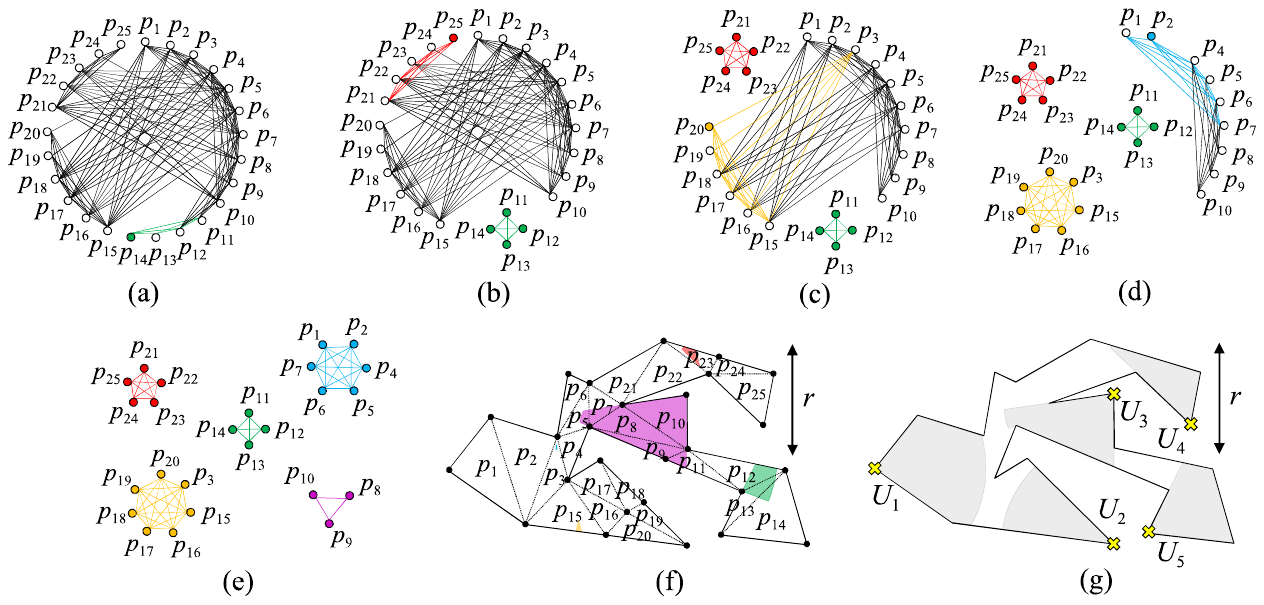}%{vtc_fig0.eps}
\vspace{-11.9cm}\caption{ Optimal deployment of VLC APs for a LoS coverage via MCC method: a) The PV graph and determining the first maximal clique; b) the second maximal clique; c) the third maximal clique; d) the forth maximal clique; e) the fifth maximal clique; f) Visibility area of the five maximal cliques that were identified in the PV graph; g) Determination of the lower-bound by disclosing a set of $s=5$ hidden points inside the layout. }\label{fig:constrained-optimal1}
\vspace{-0.5 cm}
\end{figure*}

Theorem~\ref{Theorem:clique_partitioning} suggests that finding the  minimum number of VLC APs required for LoS coverage corresponds to the PV graph partitioning into the minimum number of cliques, i.e., minimizing $g$, which is known as the minimum clique cover problem in the literature.
 Since the minimum clique cover problem is NP-hard, we propose the Maximal Clique Clustering (MCC) method shown in Algorithm \ref{tabel:Max_Clique_partitioning} to find the fewest number of cliques. By sorting the nodes in an ascending degree order, MCC determines a maximal clique by clustering the minimum degree node and adding the consequent nodes as long as it keeps forming a clique with a non-empty visibility area. Then, by removing the maximal clique, we start over the steps until no node in the PV graph is left.

 Fig.~\ref{fig:constrained-optimal1} visualizes the steps of the MCC method to identify the optimal number of VLC APs for (LoS) wireless access in the sample layout when cell range is $r$. By starting with the PV graph as in Fig.~\ref{fig-nonoptimal1.pdf}(b), we choose $p_{14}$ as the minimum degree node in Fig.~\ref{fig:constrained-optimal1}(a) and cluster $p_{11}p_{12}p_{13}p_{14}$ as the maximal clique connected to $p_{14}$. Then, by removing this clique, we select $p_{25}$ as the minimum degree node in Fig.~\ref{fig:constrained-optimal1}(b). Continuing the similar steps, we end up with the five cliques shown in Fig.~\ref{fig:constrained-optimal1}(e), whose visibility areas are displayed in Fig.~\ref{fig:constrained-optimal1}(f). As a result, deploying $g=5$ VLC APs, one inside each visibility area, ensures the full coverage of the layout. However, an effective benchmark to assess the optimality of the number of VLC APs is still demanding.

 \begin{lemma}\label{Lemma:lowerbound}
 Presume that a set of $s$ points in a layout, $U_1$, $U_2$,..., $U_s$, which have pairwise disjoint visibility areas, i.e., $\mathcal{V}(U_i)\cap \mathcal{V}(U_j)=\emptyset$ for $1\leq i,j \leq s$ and $i\neq j$, referred as a \emph{hidden} set of points. Then, $s$ specifies a lower bound to the minimum number of VLC APs required for LoS coverage of the layout, i.e., $s\leq g$.
 
  %such that $\mathcal{V}(U_i)\cap \mathcal{V}(U_j)=\O$, for $1\leq i,j \leq s$ and $i\neq j$, reveals $s$ as a lower bound to the minimum number of LEDs to fully cover the layout, i.e., $s\leq g$.
  \end{lemma}
  \begin{proof}
  To ensure LoS coverage, it is necessary to cover  the points $U_1, U_2,..., U_s$. In other words, we require to place at least one VLC AP inside each one of the disjoint areas $\mathcal{V}(U_1), \mathcal{V}(U_2),..., \mathcal{V}(U_s)$. Thus, Lemma \ref{Lemma:lowerbound} follows.
  %We skip the proof.
  %Consider the visibility areas of the hidden set of points $U_1$, $U_2$,..., and $U_s$. Each visibility area requires at least one LED to cover all the $s$ points. Thus, Lemma \ref{Lemma:lowerbound} follows.
  \end{proof}
 From  Lemma \ref{Lemma:lowerbound}, the hidden set of $s=5$ cross-marked points in Fig.~\ref{fig:constrained-optimal1}(g) confirms that the $g=s=5$ number of VLC APs determined in  Fig.~\ref{fig:constrained-optimal1}(f) is optimal. However, finding a hidden set of points might often be challenging due to the continuous nature of layouts. Thus, we alternatively utilize the PV graph according to the procedure that is detailed as follows.

% These also algorithm environment in Latex, but in this way there's no need to change the contents in the table
\renewcommand{\tablename}{Algorithm}
\begin{table}[!t]
\caption{Maximal Clique  Clustering}\label{tabel:Max_Clique_partitioning}
\vspace{-0.3cm}
\begin{center}
\hspace{0cm}\begin{tabular}{ l | l}\hline\hline

{\footnotesize{~}}&
{\footnotesize{{\footnotesize{\textbf{Input} ~~$\mathbb{G}[p_1,p_2,...,p_M]$ ~~~$\%$ The PV graph }}}} \\

\hline

{\footnotesize{1}}&
{\footnotesize{{\footnotesize{{\bf{While}} ~~~$\mathbb{G}\neq \emptyset$: }}}} \\

{\footnotesize{2}}&{\footnotesize{{\footnotesize{~~~~$\mathbb{G}[p'_1,p'_2,..., p'_M]\leftarrow$  $\texttt{Ascending-Sort}(\mathbb{G})$  }}}} \\

   {\footnotesize{3}}&
{\footnotesize{{\footnotesize{~~~~$c\leftarrow \emptyset$}}}} \\

 {\footnotesize{4}}&{\footnotesize{{\footnotesize{~~~~{\bf{For}}~~ $i=1$ \bf{to} $M$: }}}} \\

 {\footnotesize{5}}&
{\footnotesize{{\footnotesize{~~~~~~~~{\bf{If}} ~~$c\cup p'_i~ \texttt{forms a clique in~} \mathbb{G} $: }}}} \\

{\footnotesize{6}}&
{\footnotesize{{\footnotesize{~~~~~~~~~~~{\bf{If}}~~  $\mathcal{V}(c\cup p'_i)\neq \emptyset$: }}}} \\

{\footnotesize{7}} & {\footnotesize{~~~~~~~~~~~~~~~~~$c\leftarrow  c \cup p'_i$  }}\\ 

%{\footnotesize{8}}&
%{\footnotesize{{\footnotesize{~~~~~~~~ }}}} \\

{\footnotesize{8}}&
{\footnotesize{{\footnotesize{~~~~$\mathbb{G}\leftarrow \mathbb{G}-c$ }}}} \\
{\footnotesize{9}}&
{\footnotesize{{\footnotesize{~~~~$M\leftarrow |\mathbb{G}|$}}}} \\

%{\footnotesize{12}}&
%{\footnotesize{{{\footnotesize{~~~~  }}}}}\\

{\footnotesize{10}}&
{\footnotesize{{{\footnotesize{{\bf{Return}}  }}}}}\\

\hline

\end{tabular}
\end{center}
\vspace{-0.5cm}
\end{table}

   \begin{Theorem}\label{Theorem:lowerbound2}
   %To find a hidden set of points in the layout, it is enough to find an independent set of nodes in the PV graph, i.e., no two nodes of the set are adjacent. 
   
   Presume that a set of $t$ nodes $q_1$, $q_2$, ..., $q_t$ in the PV graph forms an independent set, i.e., no two nodes of the set are adjacent. If there found a hidden set of $t$  points $U_1$, $U_2$,..., $U_t$ in the layout lying inside the triangles $q_1$, $q_2$, ..., and $q_t$, respectively, then $t$ reveals a lower bound to the number of VLC APs for LoS coverage, i.e., $t \leq g$. 
   \end{Theorem}
  \begin{proof}
 Lemma \ref{Lemma:visibility_to_inside_triangle} implies that if  $\mathcal{V}(U_i)\cap \mathcal{V}(U_j)=\emptyset$, then  $\mathcal{V}(q_i)\cap \mathcal{V}(q_j)=\emptyset$ and thus $q_i$ and $q_j$ are non-adjacent in the PV graph, where $U_i$ and $U_j$ lie inside $q_i$ and $q_j$, respectively, for $1\leq i,j \leq t$. So, we can explore a hidden set of point $U_1, U_2,...,U_t$ within the layout by determining an independent set of nodes $q_1, q_2,..., q_t$. Thus, Theorem \ref{Theorem:lowerbound2} follows.
 \end{proof}
 Theorem \ref{Theorem:lowerbound2}   suggests that instead of searching for a hidden
set of points over a whole layout, it is enough to find an
independent set of nodes in the PV graph and search over
the corresponding triangles. This procedure is more straightforward due to the discrete nature of the PV graph.
To clarify Theorem \ref{Theorem:lowerbound2},  Fig.~\ref{fig:lowerbound} explains the way to determine a lower bound to $g$ through the PV graph with range $r$. Fig.~\ref{fig:lowerbound}(a)
         represents the similar PV graph as in Fig.~\ref{fig-nonoptimal1.pdf}(b), marked with an independent set of $t=5$ nodes $p_1$, $p_8$, $p_{13}$, $p_{19}$, $p_{24}$. Then, we find a hidden set of $t=5$ points $U_1$,..., $U_5$, shown by cross-marks in Fig. \ref{fig:lowerbound}(b), lying inside the triangles $p_1$, $p_8$, $p_{13}$, $p_{19}$, $p_{24}$, respectively. So, we conclude that $t=5\leq g$.

      \begin{figure}[!t]
\centering
%%%%%%%%%%%%%%%%%figure%%%%%%%%%%%%%%%%%%%%%%
\hspace{0cm}\includegraphics[width=7.8cm]{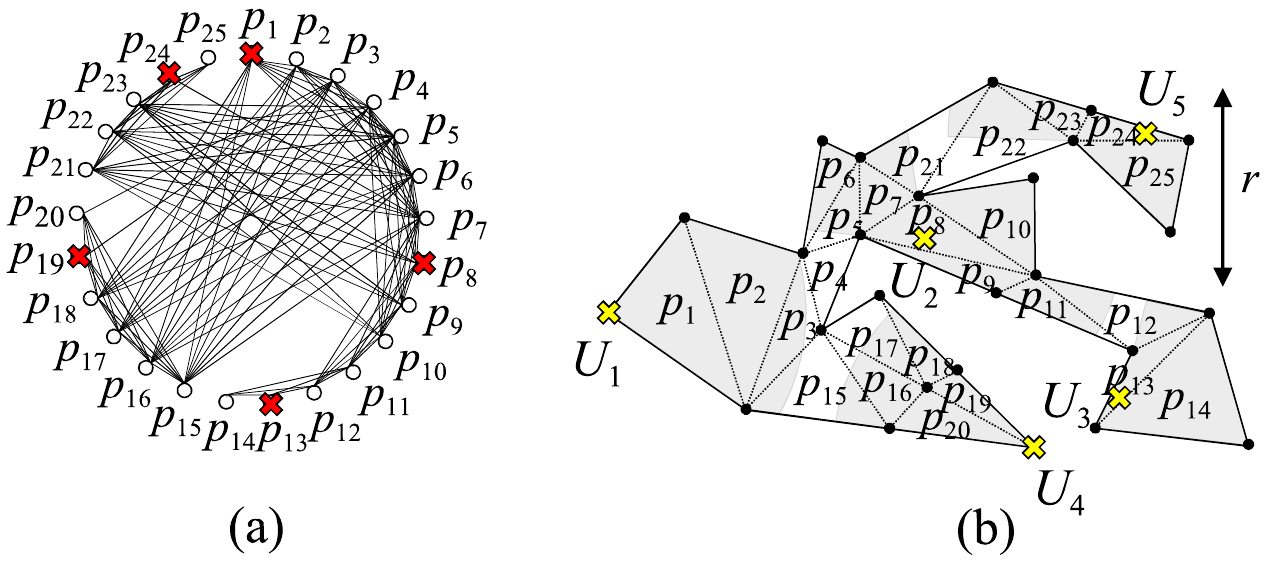}%{vtc_fig0.eps}
\vspace{-7.05cm}\caption{ Identifying a lower bound for the minimum number of VLC APs that are required for LoS coverage via the PV graph: a) An independent set of nodes with size $t=5$ in the PV graph; b) Determination of a hidden set of $t=5$ points within each corresponding triangles.}\label{fig:lowerbound}% a) An independent set of nodes in the PV graph. b) A hidden set of points lying inside the corresponding triangles.  }
\vspace{-0.5 cm}
\end{figure}   
   
        The maximum size of the hidden set of points  $s_{\max}$ is a stimulating value  that specifies a tight lower bound to $g$. However, there is no straightforward solution for $s$ maximization problem in a layout. Here again, the PV graph appears  rewarding to determine $s_{\max}$.

     \begin{Theorem}\label{theorem:lowerbound-by-pvgraph}
    Let $q_1$, $q_2$, ..., $q_{t_{\max}}$ be the largest independent set of nodes in the PV graph constructed by a  hyper triangulation from $\mathcal{HT}(R\leq 2r)$, where $t_{\max}$ denotes the independence number. If there found a hidden set of $t_{\max}$ points $U_1$, $U_2$, ..., $U_{t_{\max}}$ in the layout  lying inside the triangles $q_1$, $q_2$, ..., $q_{t_{\max}}$, respectively, then $s_{\max}=t_{\max}$. \end{Theorem}
\begin{proof} 
See Appendix.
\end{proof}
 As a result of Theorem \ref{theorem:lowerbound-by-pvgraph}, $t_{\max}$ can alternatively represent the tight lower bound to $g$. The value $t_{\max}$ in a graph is the same as the size of the largest clique in the complement graph. However, determining $t_{\max}$ is an NP-hard problem \cite{godsil2001algebraic}. %Therefore, $t_{max}$ is derived via an exhaustive search.

\begin{Remark}\label{remark:smaller-R-equality}
  The positive integer $t_{\max}$ in the PV graph  decreases non-strictly as $R$ drops in a hyper triangulation from $\mathcal{HT}(R\leq 2r)$. Thus, the equality $s_{\max}= t_{\max}$ is guaranteed  by setting a small enough $R$.
\end{Remark}
   So, if $s_{\max}=t_{\max}\leq g$ holds with  equality, then $g$ is optimal. Otherwise, making a confident statement regarding the optimality of the number of deployed VLC APs is impossible.

  \pagestyle{empty}

%\vspace{-1mm}
%\section{Art gallery problem with interconnectivity}
%\label{sec:4}
%\vspace{-0.5mm}

\vspace{0mm}
\section{Deployment of VLC APs to ensure  LoS  in wireless access and backhauling}
\label{sec:5}
\vspace{-1mm}

This section investigates the minimum number of VLC APs ($h$) required to achieve LoS wireless access coverage and wireless connectivity simultaneously. Here, wireless connectivity is defined as an additional requirement for deploying VLC APs, such that each pair of these nodes become visible to each other either directly or through intermediate VLC APs. Wireless connectivity enables optical wireless backhauling, as data can be transferred from one VLC APs to another via optical wireless technology toward the core network. Potential Deployment Areas (PDAs) are defined as regions where a VLC AP can be placed, such that the network maintains LoS coverage with connectivity.
 Therefore, any change to  one PDA may affect the others. Thus, we define a tree diagram to model the interaction among PDAs. 
 \begin{table}[!t]
\caption{\small Potential Deployment Area Updating}\label{tabel:PDA_UPDATING}
\vspace{-0.2cm}
\begin{center}
\hspace{0cm}\begin{tabular}{ l | l}\hline\hline

{\footnotesize{~}}&
{\footnotesize{{\footnotesize{\textbf{Input} ~~$\mathbb{T}[\mathcal{A}_1,\mathcal{A}_2,...,\mathcal{A}_h]$ ~~~$\%$ connectivity tree.  }}}} \\

{\footnotesize{~}}&
{\footnotesize{{\footnotesize{~~~~~~~ ~~$\mathcal{R}$ ~~~$\%$  The root node in $\mathbb{T}$.}}}} \\

\hline 

{\footnotesize{1}}&
{\footnotesize{{\footnotesize{$\mathbb{P}^1, \mathbb{P}^2,..., \mathbb{P}^l\longleftarrow$  $\texttt{Paths-to-all-leaves}(\mathbb{T},\mathcal{R})$}}}} \\&

{\footnotesize{{\footnotesize{~~~~~~~~$\%$
$\mathbb{P}^i$:  Array of nodes from $\mathcal{R}$ to the $i^{th}$ leaf in $\mathbb{T}$.}}}} \\

{\footnotesize{2}}&
{\footnotesize{{\footnotesize{{\bf{For}} ~$i=1$ \bf{to} $l$ :}}}} \\

{\footnotesize{3}}&~~~~
{\footnotesize{{\footnotesize{{\bf{For}} ~$j=1$ {\bf{to}} $|\mathbb{P}^i|-1$ :$~~\%$ $\mathcal{A}^i_j$:  $j^{th}$ node in the path 
$\mathbb{P}^i$.}}}} \\

{\footnotesize{4}}&{\footnotesize{{\footnotesize{~~~~~~~~~~~~$\mathcal{A}^i_{j+1}\longleftarrow \mathcal{CR}(\mathcal{A}^i_{j})\cap \mathcal{A}^i_{j+1}$ $~~\%$ $\mathcal{A}^i_1$  is set as $\mathcal{R}$.  }}}} \\

%{\footnotesize{5}}&~~~~
%{\footnotesize{{\footnotesize{~~~}}}} \\

{\footnotesize{5}}&
{\footnotesize{{\footnotesize{{\bf{Return}}  }}}} \\

\hline

\end{tabular}
\end{center}
\vspace{-0.5cm}
\end{table}

\begin{definition}
 A $\bf{Connectivity~Tree}~\mathbb{T}$ is a free tree in which each node represents one of the PDAs in the layout.
  Two nodes $\mathcal{A}_{a_1}$ and $\mathcal{A}_{a_i}$ in $\mathbb{T}$ are connected by a path of nodes $\mathcal{A}_{a_2}$,...., $\mathcal{A}_{a_{i-1}}$ if and only if the two VLC APs deployed in the PDAs $\mathcal{A}_{a_1}$ and $\mathcal{A}_{a_i}$ have connectivity via the array of $i-2$ VLC APs deployed in $\mathcal{A}_{a_2}$,..., $\mathcal{A}_{a_{i-1}}$.
\end{definition}

 %To maintain connectivity between the nodes of the VLC, the connectivity tree enables shrinking of the PDAs determined in the layout effectively. 
   PDA-updating method utilizes $\mathbb{T}$  to
iteratively remove infeasible areas within the PDAs, and therefore, to preserve connectivity among the VLC APs, see Algorithm \ref{tabel:PDA_UPDATING}.
A rooted version of~$\mathbb{T}$ is determined by finding the directional paths from a root node $\mathcal{R}$ to the leaves. The node $\mathcal{R}$ is interpreted as the PDA with the highest priority and degree of freedom.
Continuing from $\mathcal{R}$ to all the leaves, the PDA-updating method successively replaces the PDA at each node by the overlapped area with the connection region of the previous node.
 This way, we ensure that every point inside the resulting PDAs is feasible in terms of connectivity for a VLC AP to deploy.

%\vspace{1 cm}
\begin{figure*}[!t]
\centering
%%%%%%%%%%%%%%%%%figure%%%%%%%%%%%%%%%%%%%%%%
\hspace{0cm}\includegraphics[width=14.4cm]{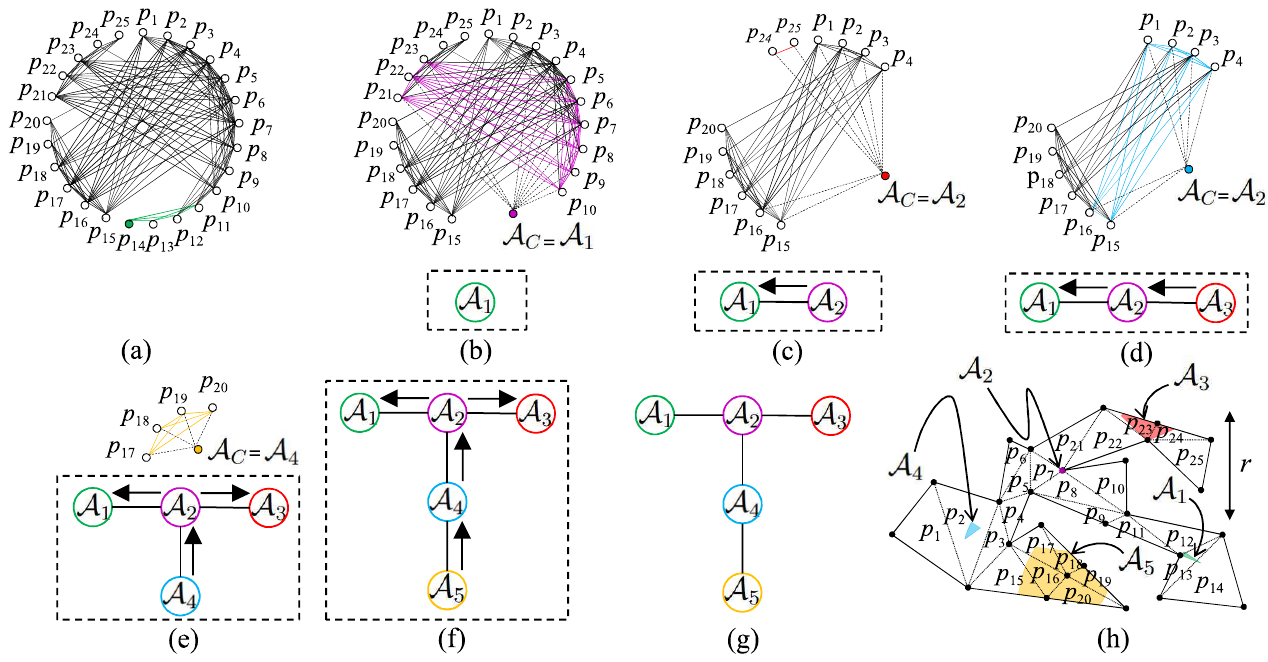}%{vtc_fig0.eps}
\vspace{-11.5cm}\caption{Procedure to find the minimum number of  VLC APs  and their PDAs to satisfy LoS coverage with connectivity via CTC method: a) Determination of first maximal clique in the PV graph as the first node $\mathcal{A}_1$ in the connectivity tree; determination of the b) second; c) third; d) forth; and e) fifth node of the connectivity tree; g) Structure of the connectivity tree; and h) Representation of the nodes in the connectivity tree (known as the PDAs).   % to determine the set of PDAs and $\mathbb{T}$ for indoor coverage with connectivity.
}\label{fig:interconnectivity}
\vspace{-0.5 cm}
\end{figure*}

   % LED deployment to satisfy indoor coverage with connectivity requires the set of PDAs as well as the structure of $\mathbb{T}$. 
   
   As shown in Algorithm~\ref{tabel:create-interconnection-tree}, the set of PDAs and $\mathbb{T}$ are simultaneously built via the Connectivity Tree Construction~(CTC) method, which utilizes the PDA-updating as an inner algorithm.
   In CTC method, the visibility area of the first maximal clique $\mathcal{V}(c)$ in the PV graph is initially assigned as the first PDA $\mathcal{A}_1$, which also forms the first node  in $\mathbb{T}$.
    Removing  $c$ from the PV graph, the rest of the method is summarized as follows: I) Among the current nodes in $\mathbb{T}$, find the node $\mathcal{A}_C$ such that $\mathcal{CR}(\mathcal{A}_C)$ has a non-empty overlap with the visibility area of the node in the PV graph with the smallest degree; II) Create the node $\mathcal{A}_C$ within the remaining PV graph and connect it to the nodes whose visibility areas have a non-empty overlap with $\mathcal{CR}(\mathcal{A}_C)$; III) While searching the nodes over the remaining PV graph in an ascending degree order, find the largest clique $c$  that is fully connected to $\mathcal{A}_C$ while satisfying $\mathcal{CR}(\mathcal{A}_C )\cap\mathcal{V}(c)\neq \emptyset$; IV) Create the new PDA $\mathcal{A}_{|\mathbb{T}|+1}:=\mathcal{CR}(\mathcal{A}_C )\cap\mathcal{V}(c)$ as a node in  $\mathbb{T}$  and connect it to  $\mathcal{A}_C$; V) Run the PDA-updating algorithm for the current $\mathbb{T}$ by assigning $\mathcal{R}:=\mathcal{A}_{|\mathbb{T}|+1}$; VI) Remove $\mathcal{A}_C$ and $c$ from the PV graph and start over until no node in the PV graph remains.

 Figure~\ref{fig:interconnectivity} illustrates the CTC steps  applied to the sample layout with cell range $r$. 
 Fig.~\ref{fig:interconnectivity}(a) shows the PV graph with the first maximal clique $c$ highlighted in green, while creating the node $\mathcal{A}_1:=\mathcal{V}(c)$ in $\mathbb{T}$  as well as creating the node $\mathcal{A}_C:=\mathcal{A}_1$ in the remaining PV graph, see Fig.~\ref{fig:interconnectivity}(b).
 Then, we connect $\mathcal{A}_C$ to the nodes whose visibility areas overlap with $\mathcal{CR}(\mathcal{A}_C)$, resulting in the discovery of the new maximal clique $c$ in purple.
 In Fig.~\ref{fig:interconnectivity}(c), we create the node $\mathcal{A}_2:=\mathcal{CR}(\mathcal{A}_1)\cap \mathcal{V}(c)\neq \emptyset$ in $\mathbb{T}$  and connect it  to the node $\mathcal{A}_C$ ($\mathcal{A}_1$ here). Now, running the PDA-updating algorithm for $\mathbb{T}$ with the root node $\mathcal{R}:=\mathcal{A}_2$, we update  $\mathcal{A}_1:=\mathcal{A}_1 \cap  \mathcal{CR}(\mathcal{A}_2)$.
 Since $\mathcal{CR}(\mathcal{A}_2)$ overlaps with the visibility area of the smallest degree node in the remaining PV graph, we create the node $\mathcal{A}_C:=\mathcal{A}_2$ as in Fig.~\ref{fig:interconnectivity}(c).
 Following the same steps, we arrive at the final $\mathbb{T}$ in Fig.~\ref{fig:interconnectivity}(g), where Fig.~\ref{fig:interconnectivity}(h) illustrates the five PDAs associated with it.

When $\mathbb{T}$ is created via the CTC method, the number of VLC APs required to provide LoS coverage and connectivity equals the number of PDAs, i.e., $h=|\mathbb{T}|$. So, deployment of $h=|\mathbb{T}|=5$ VLC APs in some of the points within the $5$ PDAs in Fig.~\ref{fig:interconnectivity}(h) is therefore sufficient for LoS coverage with connectivity.
 The following shows the steps in VLC AP deployment:
 I) Deploy a VLC AP anywhere in an arbitrary PDA from the set of PDAs $\mathcal{A}_1, \mathcal{A}_2,..., \mathcal{A}_h$. Let the PDA and the deployment point be $\mathcal{A}_i$ and $X_i$, respectively;
II)  Assign the PDA by the point $\mathcal{A}_i:=X_i$; III) Assign the root node $\mathcal{R}:=\mathcal{A}_i$ in $\mathbb{T}$ and update the PDAs using Algorithm \ref{tabel:PDA_UPDATING}; IV) Repeat the process until every PDA contains a VLC AP. Now, the set of VLC APs satisfy LoS coverage with connectivity.

%\begin{Theorem}\label{theorem:precise-vlc-locations-interconnection}
%The process of LED deployment guarantees that no PDA becomes empty in the PDA-updating algorithm.
%\end{Theorem}

%\begin{proof}

%Presuming two adjacent nodes $\mathcal{A}_j$ and $\mathcal{A}_k$ in $\mathbb{T}$ created by the CTC method, the PDA-updating algorithm guarantees that $P_j \in \mathcal{CR}(\mathcal{A}_k)$ and $P_k \in \mathcal{CR}(\mathcal{A}_j)$ for any points $P_j\in \mathcal{A}_j$ and $P_k\in \mathcal{A}_k$. As a result, Theorem \ref{theorem:precise-vlc-locations-interconnection} follows.
%\end{proof}

Presuming two adjacent nodes $\mathcal{A}_j$ and $\mathcal{A}_k$ in $\mathbb{T}$ created by the CTC, the PDA-updating algorithm guarantees that $X_j \in \mathcal{CR}(\mathcal{A}_k)$ and $X_k \in \mathcal{CR}(\mathcal{A}_j)$ for any points $X_j\in \mathcal{A}_j$ and $X_k\in \mathcal{A}_k$. Due to this, no PDA becomes empty during the deployment of VLC APs. However, a lower bound to $h$ has yet to be determined.

%It is worth noting that no PDA becomes empty during LED deployment. That is because presuming two adjacent nodes $\mathcal{A}_j$ and $\mathcal{A}_k$ in $\mathbb{T}$ created by the CTC method, the PDA-updating algorithm guarantees that $P_j \in \mathcal{CR}(\mathcal{A}_k)$ and $P_k \in \mathcal{CR}(\mathcal{A}_j)$ for any points $P_j\in \mathcal{A}_j$ and $P_k\in \mathcal{A}_k$. Thereupon, the process of LED deployment ensures indoor coverage with interconnection. However, a lower bound to $h$ is yet to  discover.

 \begin{Remark}\label{remark:interconnection-lowerbound}
   Suppose there is a set of $s$ points $U_1$, $U_2$,..., $U_s$ in a layout and a set of non-negative integers $b_1$, $b_2$,..., $b_s$ such that $\mathcal{CR}^{b_1}(\mathcal{V}(U_1))$, $\mathcal{CR}^{b_2}(\mathcal{V}(U_2))$,..., $\mathcal{CR}^{b_s}(\mathcal{V}(U_s))$ form a set of pairwise disjoint areas, where $\mathcal{CR}^{b_i}(\mathcal{V}(U_i))\equiv \mathcal{CR}(...\mathcal{CR}(\mathcal{CR}(\mathcal{V}(U_i))...)$ denotes a $b_i$ times composite function. Then, the inequality $s+\sum_{i=1}^s b_i \leq h $  indicates a lower bound to the number of VLC APs that satisfies LoS coverage with connectivity.
 \end{Remark}

From Remark \ref{remark:interconnection-lowerbound}, we can derive the tight lower bound to $h$ by maximizing $s+\sum_{i=1}^s b_i $ over $s$, location points $U_1$, $U_2$,..., $U_s$, as well as $b_1$, $b_2$,..., $b_s$. 
 Hence, Remark \ref{remark:interconnection-lowerbound} can be viewed as a
  special case of  Lemma \ref{Lemma:lowerbound}, wherein $b_i=0$ and therefore $\mathcal{CR}^{0}(\mathcal{V}(U_i))=\mathcal{V}(U_i)$ for all $i$.

 \begin{table}[t!]
\caption{ Connectivity Tree Construction}\label{tabel:create-interconnection-tree}
\vspace{-0.5cm}
\begin{center}
\hspace{0cm}\begin{tabular}{ l | l}\hline\hline

%{\footnotesize{~}}&
%{\footnotesize{{\footnotesize{\textbf{Input} ~~$\mathbb{G}[p_1,p_2,...,p_M]$ ~~~$\%$ The PV graph   }}}} \\

%\hline

{\footnotesize{~}}&
{\footnotesize{{\footnotesize{\textbf{Initialize}~~~~$\mathbb{T}\longleftarrow \emptyset$ ~~~~~~~~$\%$  Connectivity tree }}}} \\

{\footnotesize{~}}&{\footnotesize{{\footnotesize{~~~~~~~~~~~~~~~$c $  $~~~~~\%$  The first maximal clique in $\mathbb{G}$, Table \ref{tabel:Max_Clique_partitioning}   }}}} \\
\hline

%{\footnotesize{1}}&{\footnotesize{{\footnotesize{$c \longleftarrow$ \texttt{First maximal clique in} $\mathbb{G}$ $~\%$  Table \ref{tabel:Max_Clique_partitioning}   }}}} \\

{\footnotesize{1}}&{\footnotesize{{\footnotesize{$\mathcal{A}_1\longleftarrow \mathcal{V}(c)$  $~~~~~~~~~~~~~~~~~~\%$  The first PDA }}}} \\

{\footnotesize{2}}&{\footnotesize{{\texttt{Create the first node $\mathcal{A}_1$ in $\mathbb{T}$ } }}} \\

{\footnotesize{3}}&{\footnotesize{{\footnotesize{$\mathbb{G}\longleftarrow \mathbb{G} - c$    }}}} \\

{\footnotesize{4}}&{\footnotesize{{\footnotesize{$M\longleftarrow |\mathbb{G}|$     }}}} \\

{\footnotesize{5}}&
{\footnotesize{{\footnotesize{{\bf{While}} ~~$\mathbb{G}\neq \emptyset$: }}}} \\

 {\footnotesize{6}}&{\footnotesize{{\footnotesize{~~~~$\mathbb{G}[p'_1,p'_2,..., p'_M]\leftarrow$  $\texttt{Asc-Sort}(\mathbb{G})$  }}}} \\
 
 {\footnotesize{7}}&{\footnotesize{{\footnotesize{~~~~$k\longleftarrow \text{min}~~~ i, ~~ s.t. ~~\mathcal{V}(p'_i) \cap \big(\bigcup_{\mathcal{A}_j\in \mathbb{T}}\mathcal{CR}(\mathcal{A}_j)\big)\neq \emptyset $ }}}} \\

{\footnotesize{8}}&{\footnotesize{{\footnotesize{~~~~$\mathcal{A}_C\longleftarrow \underset{\mathcal{A}_j \in \mathbb{T}}{\text{argmax}}~ \sum_{i=1}^M sgn(|\mathcal{V}(p'_k)\cap\mathcal{CR}(\mathcal{A}_j)\cap \mathcal{V}(p'_i)|)$ }}}} \\

{\footnotesize{9}}&{\footnotesize{{\footnotesize{~~~ \texttt{Create the node $\mathcal{A}_C$ in $\mathbb{G}$}  }}}} \\

{\footnotesize{10}}&
{\footnotesize{{\footnotesize{{~~~~{\bf{For}}~~$i=1$\texttt{ to} $M$:}}}}} \\

{\footnotesize{11}}&{\footnotesize{{\footnotesize{~~~~~~~ \bf{If}~~$\mathcal{CR}(\mathcal{A}_C)\cap \mathcal{V}(p'_i)\neq \emptyset$}:}}} \\

   {\footnotesize{12}}&
{\footnotesize{{\footnotesize{~~~~~~~~~~~~ $\texttt{Con}(\mathbb{G},p'_i,\mathcal{A}_C)$ ~~$\%$ Connect $p'_i$ and $\mathcal{A}_C$ in $\mathbb{G}$ }}}} \\

%{\footnotesize{~}}&
%{\footnotesize{{\footnotesize{~~~~~~~~~~~~  }}}} \\

% {\footnotesize{11}}&{\footnotesize{{\footnotesize{~~~~$\mathbb{G}[p'_1,p'_2,..., p'_m]\leftarrow$  $\texttt{Asc-Sort}(\mathbb{G})$  }}}} \\
 
 {\footnotesize{13}}&
{\footnotesize{{\footnotesize{~~~~$c\longleftarrow \emptyset$ }}}} \\

 {\footnotesize{14}}&{\footnotesize{{\footnotesize{~~~~{\bf{For}}~~ $i=1$\texttt{ to} $M$: }}}} \\

 {\footnotesize{15}}&
{\footnotesize{{\footnotesize{~~~~~~~~{\bf{If}}~~\texttt{$\mathcal{A}_C$, $p'_i$, and $c$ form a clique}:}}}} \\

 {\footnotesize{16}}&
{\footnotesize{{\footnotesize{~~~~~~~~~~~~{\bf{If}}~~  $\mathcal{CR}(\mathcal{A}_C )\cap \mathcal{V}(p'_i) \cap  \mathcal{V}(c)\neq \emptyset$:}}}} \\

{\footnotesize{17}} & {\footnotesize{~~~~~~~~~~~~~~~~~$c\leftarrow  c \cup p'_i$  }}\\ 

%{\footnotesize{~}} & {\footnotesize{~~~~~~~~~~~~~~~~~ }}\\ 
%{\footnotesize{8}}&
%{\footnotesize{{\footnotesize{~~~~~~~~~~~~ }}}} \\

{\footnotesize{18}}&
{\footnotesize{{\footnotesize{~~~  $\mathcal{A}_{|\mathbb{T}|+1}\longleftarrow \mathcal{CR}(\mathcal{A}_C )\cap\mathcal{V}(c)$   }}}} \\

 {\footnotesize{19}}&
{\footnotesize{{\footnotesize{~~~ 
\texttt{Create the node $\mathcal{A}_{|\mathbb{T}|+1}$ in $\mathbb{T}$ }}}}} \\

 {\footnotesize{20}}&
{\footnotesize{{\footnotesize{~~~  $\texttt{Con}(\mathbb{T},\mathcal{A}_{|\mathbb{T}|+1},\mathcal{A}_C) $  }}}} \\

{\footnotesize{21}}&
{\footnotesize{{\footnotesize{~~~~$\mathcal{A}_{1},...,  \mathcal{A}_{|\mathbb{T}|}\longleftarrow \texttt{PDA-Update}(\mathbb{T},\mathcal{R}=\mathcal{A}_{|\mathbb{T}|+1})$~$\%$ Table \ref{tabel:PDA_UPDATING} }}}} \\

%{\footnotesize{~}}&
%{\footnotesize{{\footnotesize{~~~~~~~~~~~~~~~~~~~~~~~~~~~~~~~~~~~~~~$\%$ Table \ref{tabel:PDA_UPDATING} }}}} \\

{\footnotesize{22}}&
{\footnotesize{{\footnotesize{~~~~$\mathbb{G}\leftarrow \mathbb{G}-(\mathcal{A}_C\cup c)$ }}}} \\

{\footnotesize{23}}&
{\footnotesize{{\footnotesize{~~~~$M\leftarrow |\mathbb{G}|$}}}} \\

%{\footnotesize{12}}&
%{\footnotesize{{{\footnotesize{{\bf{~}}  }}}}}\\

{\footnotesize{24}}&
{\footnotesize{{{\footnotesize{{\bf{Return}}  }}}}}\\

\hline

\end{tabular}

\end{center}
\vspace{-3mm}
\end{table}

\vspace{0mm}
\section{Simulation results}
\label{sec:6}
\vspace{-0mm}

%\vspace{-1mm}
%\subsection{Determination of the range of the cells using vlc link level parameters}
%\label{sec:4}
%\vspace{-0.5mm}

%\vspace{-1mm}
%\subsection{Performance figures}
%\label{sec:4}
%\vspace{-0.5mm}

To assess the effectiveness of the proposed deployment method for optical wireless networks, we first present various simulation scenarios along with their corresponding parameter values. Next, we utilize the proposed algorithms to generate a visual representation of the optimal number of VLC APs and their optimal locations over an actual floor plan. Finally, we evaluate the performance of the proposed algorithms with respect to data rates and illumination levels.

\subsection{Simulation scenarios and parameters}
To compare our results and those obtained by conventional methods with long ($r=10$ m), medium ($r=3$ m), and short ($r=2$ m) maximum ranges, we applied all three methods to the ground floor plan of the Jewish Museum in London, which contains $n=67$ vertices arranged within an area of approximately $30 \times 35$\,m, see Fig.~\ref{fig:simulation-coverge-vs-connectivity}. We employ the values for the VLC network parameters presented in Table \ref{Table:parameter-values}.
  % \begin{itemize}
   
  %      \item LED transmission bandwidth $B=10$\,MHz .
        
  %     \item Semi-angle at half power of the PC-LED $\theta_{\text{max}}=60\degree$.
       
   %    \item Active area of the PD $A_{\text{pd}}=75.4$\,mm$^2$ .
   %    \item Vertical distance from LEDs to the PD $h_{\text{LED-PD}}=2.5$\,m.
   %    \item PD facing upwards with the Field of View (FoV) semi-angle   ${\Psi}=\text{tan}^{-1}(r/h_{\text{LED-PD}})$.
   %    \item Noise equivalent power $\text{NEP}=3.17\times 10^{-10}$\,W/$\sqrt{\text{Hz}}$.

    %   \item The area under $S_o^{(\text{W})}$ is normalized to unit.
   %\end{itemize}

   To analyze data rate and illumination key performance indicators, we 
   average over randomly placed user locations in the museum layout.
   %consider a single user located anywhere in the museum layout with a uniform 2-dimensional probability distribution. 
   Moreover, we assume two power allocation schemes in the VLC system. In the first scheme, each VLC AP emits with a power of $P_{\rm LED} = 40$\,W, whereas in the second scheme all VLC APs share evenly a total power of $P_{\rm tot} = 1$\,kW. These values only serve to compare the performance and are not suggested for real implementations. We also assume that any pair of VLC APs with a mutual LoS condition can communicate via an optical wireless backhaul link at distances beyond the size of the museum. 
   %Therefore, we assume that the range of optical wireless backhauling is unlimited.

  % To fairly analyze data rate and illumination performances, we  consider two power allocation schemes in the VLC channel; First, each LED emits with the power of $P_{\text{LED}}=40$\,W, and second, all LEDs evenly share a total power of $P_{\text{tot}}=1$\,kW. We also presume that any two LEDs with LoS condition may communicate via optical wireless backhauling at distances beyond the size of the museum. Therefore, the range of optical wireless backhauling is assumed to be 
  % unlimited.

\begin{figure*}[!t]
\centering
\advance\leftskip-1.8cm
\advance\rightskip-1cm
%%%%%%%%%%%%%%%%%figure%%%%%%%%%%%%%%%%%%%%%%
\includegraphics[width=19cm]{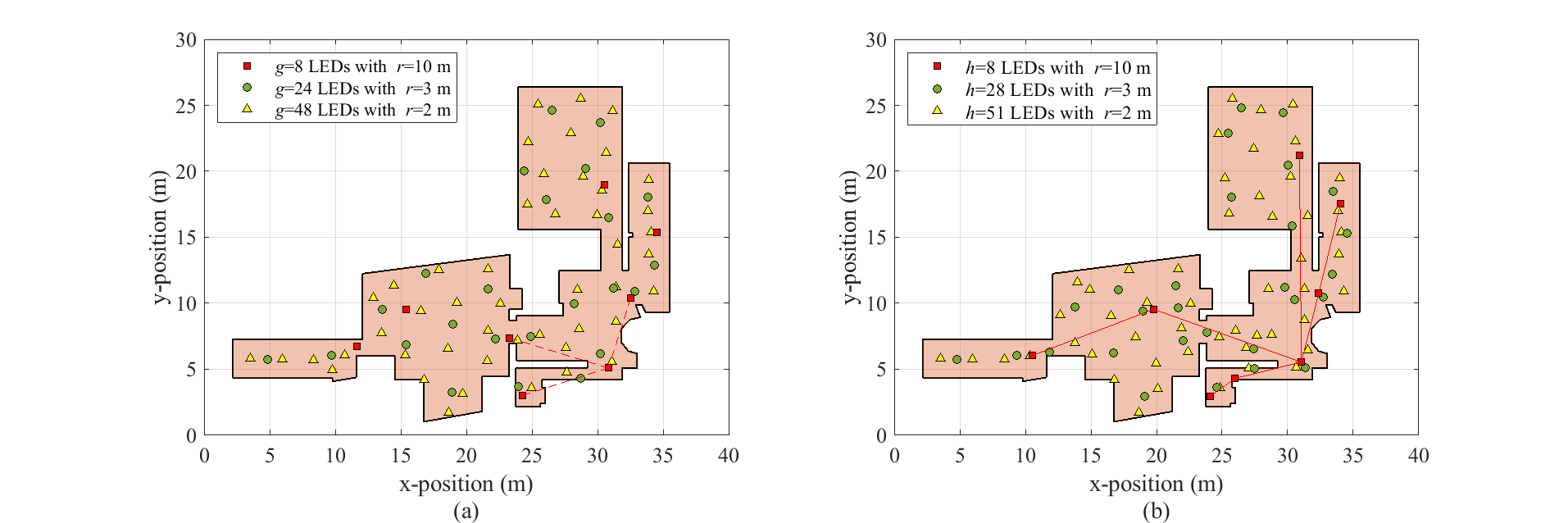}%{vtc_fig0.eps}
\vspace{-0.3cm}\caption{  Deployments of LED-based VLC APs for the given museum layout  with three maximum ranges to satisfy: a) LoS coverage via MCC deployment; b) LoS coverage with connectivity via CTC method. }\label{fig:simulation-coverge-vs-connectivity}
\vspace{-0.5 cm}
\end{figure*}

\begin{figure*}[!t]
\centering
\advance\leftskip-1.8cm
\advance\rightskip-1cm
%%%%%%%%%%%%%%%%%figure%%%%%%%%%%%%%%%%%%%%%%
\includegraphics[width=19cm]{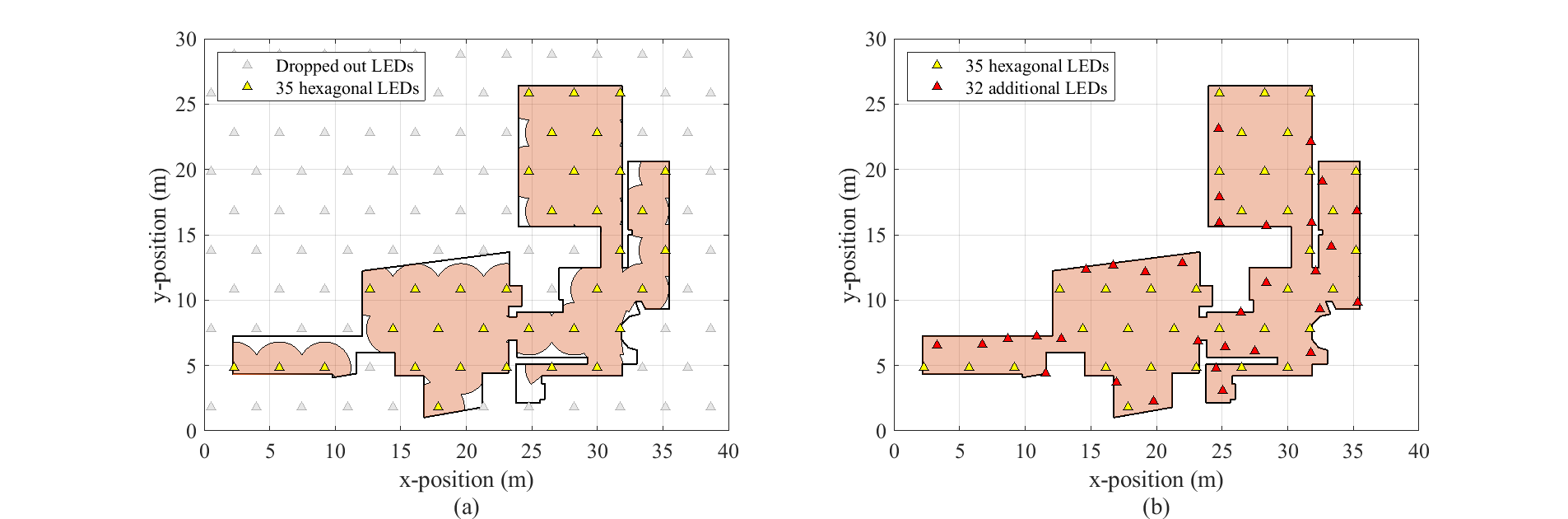}%{vtc_fig0.eps}
\vspace{-0.3cm}\caption{  Conventional LED-based VLC AP deployments applied to the museum layout with the maximum cell range $r=2$\,m via: a) HEX deployment; and b) HEX+ deployment with LoS coverage. }\label{fig:simulation-conventional}
\vspace{-0.4 cm}
\end{figure*}

\subsection{Statistical analysis of the number of VLC APs required in different deployment strategies}
 Figure~\ref{fig:simulation-coverge-vs-connectivity}(a) illustrates the deployment of the minimum number of VLC APs required for LoS coverage of the museum via the MCC method.
With $r=10$\,m, the MCC method suggests $g=8$ VLC APs, shown in red squares in the locations of the center points of the eight resulting clique visibility areas.
The same procedure is applied when reducing the range to $r=3$  and $2$\,m, with the result being $g=24$ and $g=48$ VLC APs represented by green circles and yellow triangles, respectively.
Hence, more VLC APs are observed with a smaller cell range to ensure LoS coverage. In spite of this, the dashed red lines between the four red squares show that the MCC method does not guarantee connectivity among all VLC APs. Fig.~ \ref{fig:simulation-coverge-vs-connectivity}(b) represents the results of the CTC method in determining the minimum number of VLC APs and their placement for LoS coverage with connectivity.
By setting $r=10$\,m, we obtain a total of $h=8$ VLC APs covering the entire area and maintaining connectivity (solid red lines). 
Furthermore, the CTC method allocates $h=28$ and $h=51$ VLC APs when $r=3$\,m  and $r=2$\,m, respectively.
 So, there is no significant difference between MCC and CTC regarding the number of VLC APs. It is also observed that relocation of the VLC APs is sometimes sufficient to satisfy LoS coverage and connectivity, such as in the case of cell range $r=10$\,m.

   \begin{figure}[!t]
%\centering
%%%%%%%%%%%%%%%%%figure%%%%%%%%%%%%%%%%%%%%%%
\vspace{-0.2cm}\includegraphics[width=8cm]{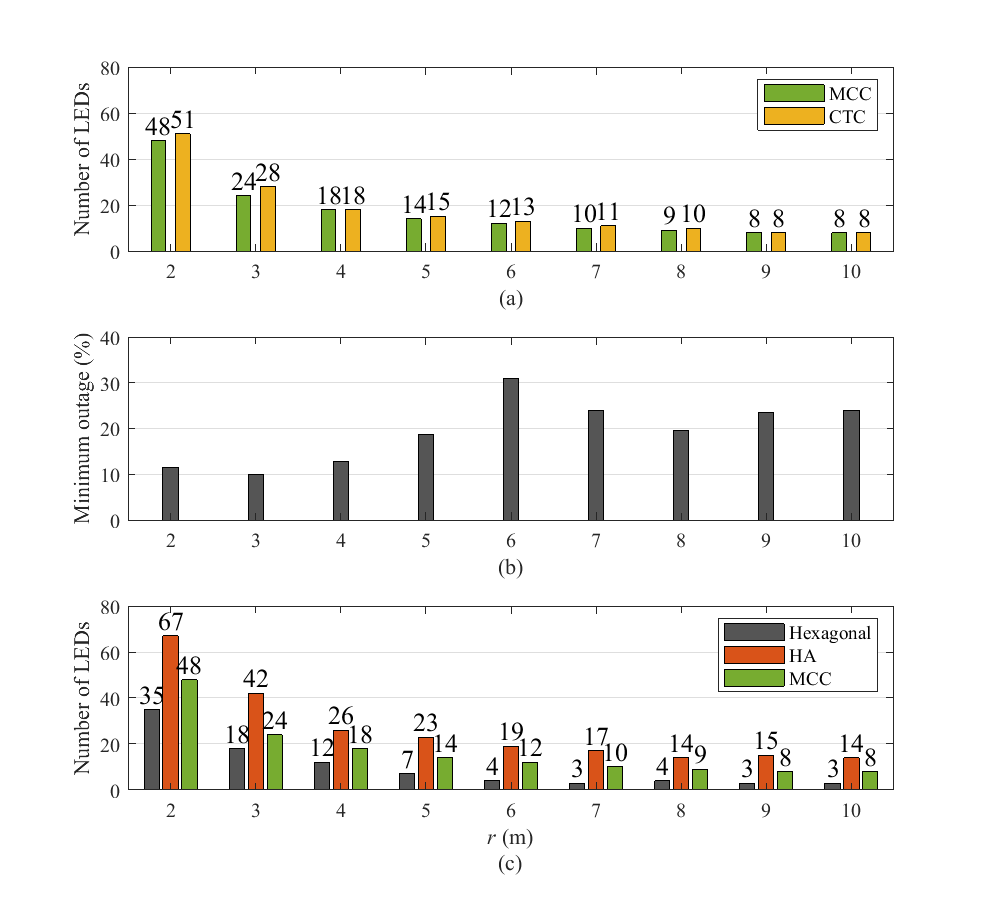}%{vtc_fig0.eps}
\vspace{-0.5cm}\caption{  Performance evaluation of the LED-based VLC AP deployment methods applied to the museum layout: a) Comparison of the number of VLC APs resulting from MCC and CTC; b) minimum outage resulting from HEX deployment; and c) comparison of the number of VLC APs resulting from the HEX, HEX+, and MCC deployments. }\label{fig:simulation-bargraphs}\vspace{0cm}%%%%% remove this vspace later
\vspace{-0.3 cm}
\end{figure}

 Next, we demonstrate two representative conventional indoor VLC AP deployments, in which the \emph{HEX deployment} (see Fig.~\ref{fig:squar-hexagonal}) refers to the method of placing the layout over the hexagonal cells such that the placement of VLC APs minimize the indoor areas in outage. Here, Fig.~\ref{fig:simulation-conventional}(a) illustrates the HEX deployment with $r=2$\,m, in which $35$ VLC APs shown by yellow triangles are located in the indoor service area, while the gray triangles indicate the VLC APs outside the area. %that were  dropped out.
  The white areas within the layout mark the outage zones.

  \setcounter{table}{0}  %%%%% this command start the table numbers from 1
\vspace{0cm}
\renewcommand{\tablename}{Table} 
\begin{table}
\centering
\advance\leftskip-1cm
\advance\rightskip-1cm
\caption{List of parameter values }\label{Table:parameter-values}
\begin{tabular}{ |p{0.8cm}||p{3.7cm}|p{1.4cm}| p{1.1cm}| }
 \hline
 %\multicolumn{3}{|c|}{Country List} \\
 \hline
 \small{Symbol} & ~~~~~~~~ ~~~~\small{Name}  & ~~~\small{Value} &~~ \small{Unit}\\
 \hline
~~\small{$B$}   &  \small{LED transmission bandwidth}   & ~~~~\small{10} & \small{~[MHz]}\\
~~~\small{$\theta_{\max}$}&   \small{Semi-angle of the PC-LED} & ~~~\small{$\pi/3$} & ~~\small{[rad]}   \\
~~\small{$A_{\text{pd}}$} & \small{Active area of the PD}  & ~~~\small{75.44} & \small{~~[mm$^2$]}\\
\small{$h_{\text{LED-PD}}$}  & \small{LED to PD vertical distance} & ~~~~$2.5$ & ~~~\small{[m]}\\
 %American Samoa&   FoV semi-angle of the PD   ${\Psi}=\text{tan}^{-1}(r/h_{\text{LED-PD}})$.  & ASM\\
~~\small{NEP} & \small{Noise equivalent power}   & \small{\mbox{$3.17{\rm{e-10}}$}} & \small{[W/$\sqrt{\text{Hz}}$]}  % \\
 %~~ $S_o^{(\text{W})}$ &  LED power spectral distribution  & ~~~$\dfrac{1}{\lambda_u-\lambda_l}$ & ~~~[1/m]
 \\
 \hline
\end{tabular}
\vspace{-5mm}
\end{table}

 In the \emph{HEX+ deployment} method, the HEX-deployed VLC APs are coupled with the fewest additional VLC nodes to cover the entire layout. This is accomplished by deploying each additional VLC AP in such a way that it covers the largest portion of the outage zone. %area that is uncovered. 
 In Fig.~\ref{fig:simulation-conventional}(b), the $32$ red triangles represent additional VLC APs, on top of the HEX-deployed ones, that jointly provide LoS coverage on the whole museum layout.
  Consequently, $35+32=67$ VLC APs with $r=2$\,m are required for LoS coverage when using HEX+.
  As a result, the HEX deployment leaves a large number of zones uncovered, while HEX+ significantly increases the total number of VLC APs. Later, we demonstrate that the interference caused by additional VLC APs degrades the data rate of users.

Figure~\ref{fig:simulation-bargraphs} compares the performance of the  deployment methods when the cell range varies between $r=2$\,m and $10$\,m.
  Fig. \ref{fig:simulation-bargraphs}(a) shows that, compared to the MCC, CTC deployment does not significantly increase the number of VLC APs. Meanwhile, Fig. \ref{fig:simulation-bargraphs}(b) represents the minimum outage percentage of the HEX deployment, where it is observed that the outage probability grows with $r$, and can be as high as $30\%$ with $r=6$\,m.
  Fig.~\ref{fig:simulation-bargraphs}(c) illustrates the number of VLC APs for the three deployment methods. Although HEX requires fewer \mbox{VLC APs}, it does not cover the entire layout. Finally, it is observed that the MCC method reduces the number of required VLC APs by up to $50\%$, as compared to HEX+. %This is  which shows the advantage of the proposed method over the conventional ones. 

In Figure~\ref{fig:simulation-random-layout}, we compare all the deployment methods in $100$ random layouts with $n=100$ vertices each within a $30 \times 30$\,m  square area. To create random layouts, we used \emph{inward denting}, as described in \cite{gewali2015constructing}. 
In this method, a Delaunay triangulation is constructed by randomly generating a large number of points and forming a convex polygon at the boundaries. Then, if the polygon has less than $n$ edges, we randomly remove a triangular facet whose sides share only one edge with the boundaries. This increases the number of boundary edges. Otherwise, we randomly pick a triangular facet that shares only one edge with the rest of the triangulation to reduce the number of boundary edges. We continue this procedure to achieve $n$ edges.

  \begin{figure}[!t]
\centering
%%%%%%%%%%%%%%%%%figure%%%%%%%%%%%%%%%%%%%%%%
\hspace{-0.5cm}\includegraphics[width=8cm]{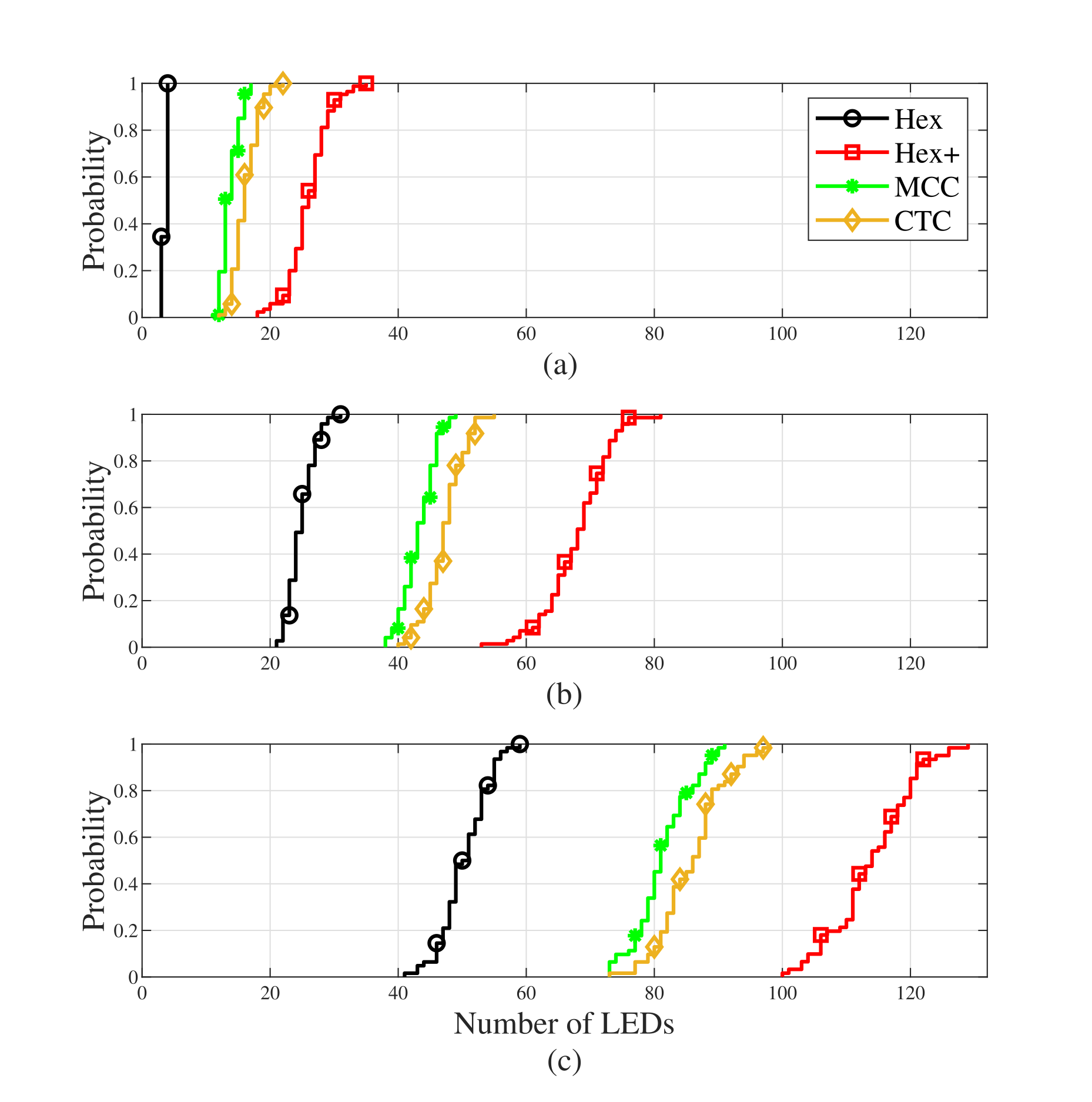}%{vtc_fig0.eps}
\vspace{-0.5cm}\caption{   Cumulative distribution functions of the number of VLC APs, applied to $100$ random layouts, via different deployment methods with maximum ranges, namely: a) $r=10$\,m; b) $r=3$\,m; and c) $r=2$\,m. }\label{fig:simulation-random-layout}
\vspace{-0.5 cm}
\end{figure}

When using the HEX deployment, as in Fig.~\ref{fig:simulation-random-layout}(a) with $r=10$\,m, three or four VLC APs fall into the layout.
In HEX+ deployment, the number of required VLC APs varies between $18$ and $36$, while the MCC method shows $g=12$ to $g=18$ VLC APs, indicating a decrease of $40\%$ on average.
The CTC method, on the other hand, requires a slightly larger number of VLC APs than the MCC method to satisfy LoS coverage with connectivity. A similar pattern is observed in Fig.~\ref{fig:simulation-random-layout}(b) and Fig.~\ref{fig:simulation-random-layout}(c), with a larger number of VLC APs when $r=3$\,m and $r=2$\,m are used, respectively. 
As a result, the MCC and CTC deployment methods statistically result in similar numbers of VLC APs which are considerably lower than the ones using HEX+ method.

   \begin{figure*}[!t]
\centering
\advance\leftskip-1cm
\advance\rightskip-1cm
%%%%%%%%%%%%%%%%%figure%%%%%%%%%%%%%%%%%%%%%%
\includegraphics[width=15cm]{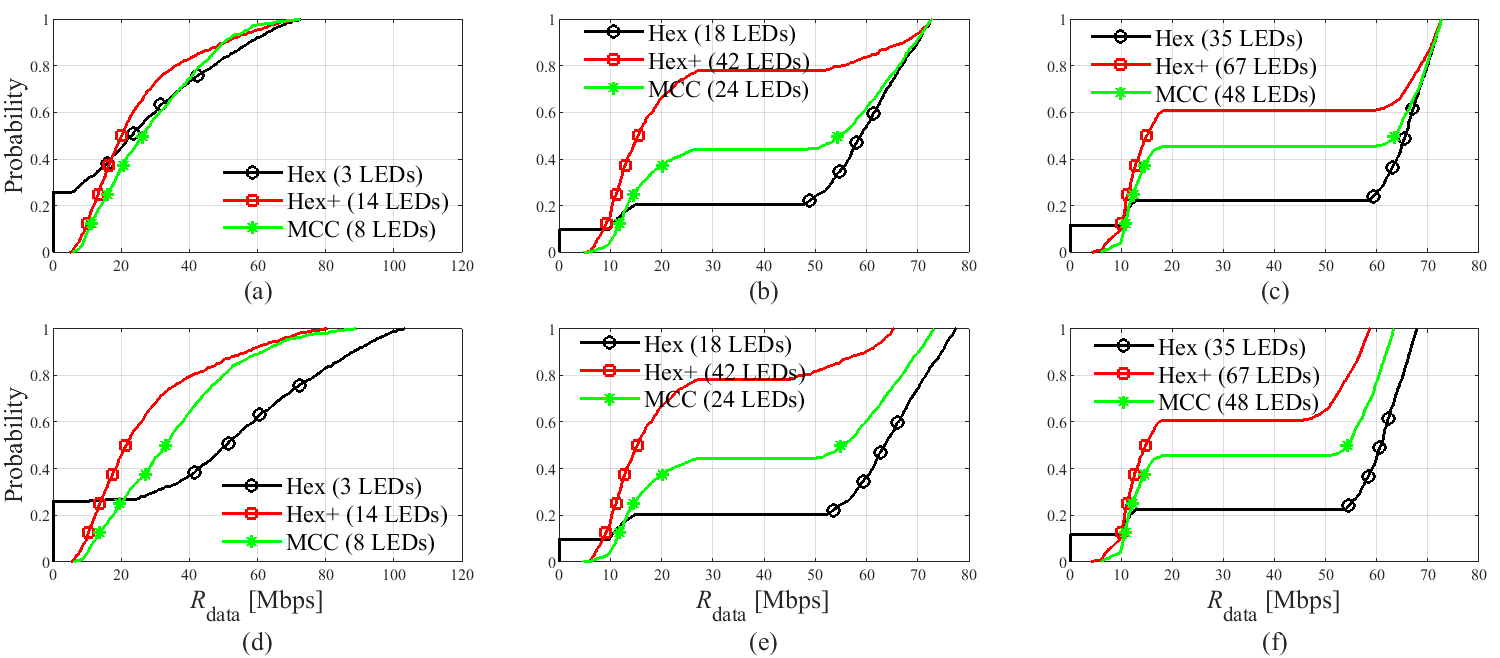}%{vtc_fig0.eps}
\vspace{-0.35cm}\caption{
 Data rate CDFs of studied deployment methods applied to the sample museum layout. Fixed LED power $P_{\rm LED}=40$\,W with VLC cell range: a) $r=10$\,m; b) $r=3$\,m; and c) $r=2$\,m. Fixed total LED power $P_{\rm tot}=1000$\,W with VLC cell range: d) $r=10$\,m; e) $r=3$\,m; and f) $r=2$\,m. }\label{fig:simulation-SINR}
\vspace{-0.2 cm}
\end{figure*} 

  \subsection{Statistical analysis of data rates and illumination provided in different deployment strategies}
  
Figure \ref{fig:simulation-SINR} shows the CDF plots of data rates, using (\ref{eq:data-rate}), for a user (PD) with uniformly distributed locations within the indoor service area defined by the layout.
    Fig. \ref{fig:simulation-SINR}(a)-(c) illustrate the CDF of the data rate for three deployment methods with a fixed LED power $P_{\rm LED}=40$\,W and $r=10$\,m, $r=3$\,m, and $r=2$\,m, respectively.
It is observed in Fig.~\ref{fig:simulation-SINR}(a) that the  eight \mbox{VLC APs} that result from MCC outperform the fourteen VLC APs from HEX+ by around $20\%$.  
The HEX deployment can yield high data rates, but then almost $25\%$ of the museum's area is in outage.
%In spite of this, about $25\%$ of the museum's area is unable to receive any data rate with HEX deployment. 
In addition, it can be seen that the MCC method outperforms HEX+ method with less VLC APs for $r=3$\,m and $r=2$\,m.
In these cases, although the HEX method improves the mean data rates, about $12\%$ of the museum area is in outage.

    The flat segments in the CDFs originate from the notable differences in data rates observed between the locations where a user does not experience interference (i.e., there is LoS condition to only one VLC AP) and the ones that experience interference (LoS condition to more than one VLC APs).
    According to Fig.~\ref{fig:simulation-SINR}(b), a user in approximately $80\%$ of the museum area receives interference with HEX+ deployment, while the value is reduced to $40\%$ for MCC.
    Then, when setting a fixed total power $P_{\rm tot}=1000$\,W, Fig.~\ref{fig:simulation-SINR}(d)-(f) show similar results for the three cell ranges under analysis.

     \begin{figure*}[!t]
\centering
\advance\leftskip 0cm
\advance\rightskip-1.2cm
%%%%%%%%%%%%%%%%%figure%%%%%%%%%%%%%%%%%%%%%%
\hspace{-1cm}\includegraphics[width=15.2cm]{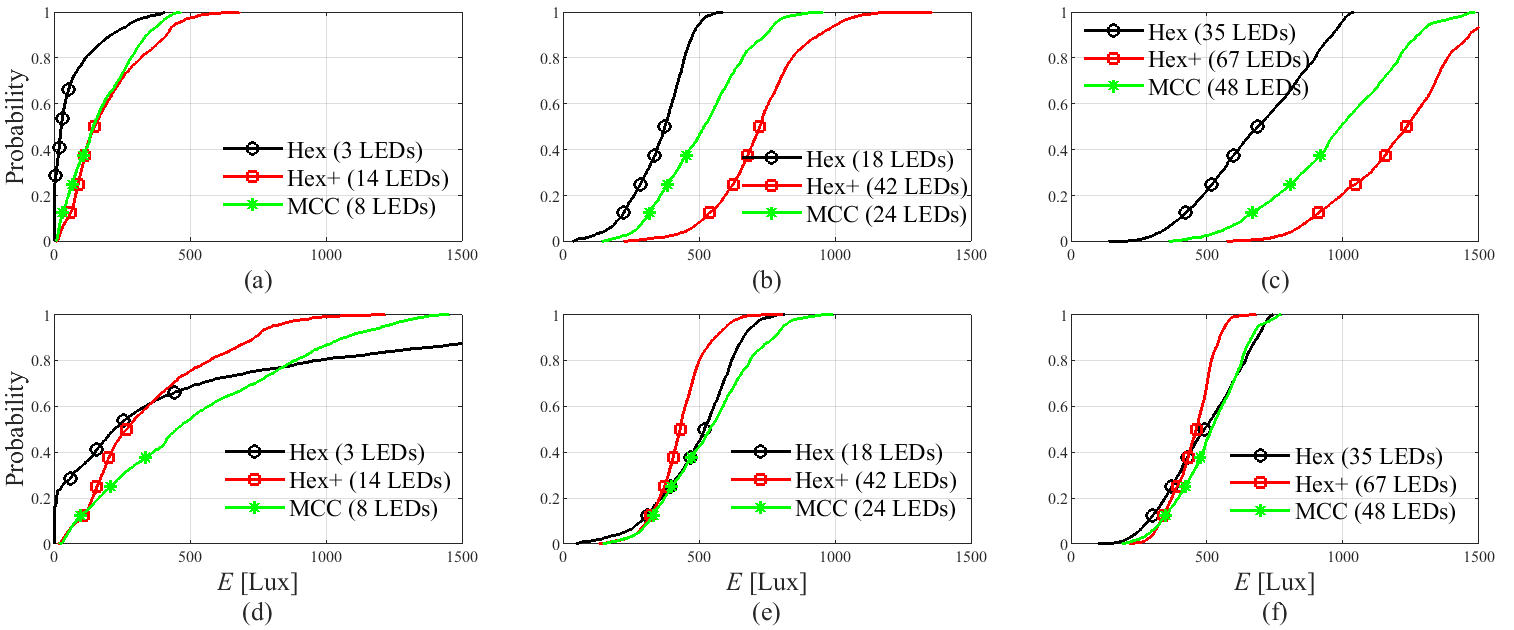}%{vtc_fig0.eps}
\vspace{-0.3cm}\caption{ Illumination CDFs of studied deployment methods applied to the sample museum layout. Fixed LED power $P_{\rm LED}=40$\,W with VLC cell range: a) $r=10$\,m; b) $r=3$\,m; and c) $r=2$\,m. Fixed total LED power $P_{\rm tot}=1000$\,W with VLC cell range: d) $r=10$\,m; e) $r=3$\,m; and f) $r=2$\,m. }\label{fig:simulation-Illumination}
\vspace{-0.4 cm}
\end{figure*}

    In Figure~\ref{fig:simulation-Illumination}, the effects of the different deployment methods are illustrated from the perspective of illumination, wherein the        VLC APs provide additive contribution in terms of illumination at all locations with a LoS condition using (\ref{eq:illumination}).
       By employing fixed LED power $P_{\rm LED}=40$\,W, Fig. \ref{fig:simulation-Illumination}(a)-(c) indicate a stronger illumination when using HEX+ as a result of the larger number of VLC APs. 
        When employing fixed total LED power $P_{\rm tot}=1000$\,W, the deployment methods presented in Fig.~\ref{fig:simulation-Illumination}(d)-(f) do not significantly differ in average illumination for the three VLC cell ranges under analysis. %, due to the total power consumption is kept fixed.

     Based on the CDF plots in Fig.~\ref{fig:simulation-SINR} and Fig.~\ref{fig:simulation-Illumination}, we collect the representative data rates and illumination metrics in Table \ref{Table:Data_rate_illumination}. Here, $R_{\rm mean }$ and $E_{\rm avg}$ are defined as the average data rate and average illumination received by the user, respectively. 
     User experienced minimum data rate $R_{\rm min}$ is defined as $5\%$ percentile point of user throughput according to ITU~\cite{ITU2017}.
     In addition, we define the illuminance uniformity as the $10$-th percentile of illumination $E_{0.1}$ over the average illumination, $U_\text{E}=E_{0.1}/E_\text{avg}$, see\cite{dowh2020a}.
       Here, HEX deployment is observed to result in a zero minimum data rate as a result of the outage, as well as a low level of illumination throughout the museum due to the small number of VLC APs. Because there is a small overlap in the coverage areas of the VLC cells, interference is minimized, and thus, Hex outperforms the other deployments in terms of average data rate.
       HEX+ deployment, on the other hand, improves the illumination but deteriorates the data rate metrics as a result of large overlapping areas between VLC cells, which leads to stronger co-channel interference.
    Finally, we observe that the MCC deployment avoids zero minimum data rates while maintaining a moderate level of average data rate and illumination metrics. 
    We observe that the MCC method outperforms HEX+, in particular when it comes to data rate metrics, regardless of the maximum range and the power allocation scheme. 
    Finally, we observe that given a network of LEDs, the choice of power allocation scheme has a significant impact on the average illumination of the area, but has no significant influence on data rate metrics.
       
        % This is indeed a table, and not an algorithm
\renewcommand{\tablename}{Table}
\begin{table*}
\centering
\caption{ Data rates and illumination performances in the museum layout for conventional and proposed VLC AP deployment methods. \vspace{-0.2cm}}\label{Table:Data_rate_illumination}

\begin{tabular}{|p{.039\textwidth}|p{.12\textwidth}||p{.035\textwidth}|p{.035\textwidth}|p{.035\textwidth}|p{.035\textwidth}||p{.035\textwidth}|p{.035\textwidth}|p{.035\textwidth}|p{.035\textwidth}||p{.035\textwidth}|p{.035\textwidth}|p{.035\textwidth}|p{.035\textwidth}|}
\hline\multirow{3}{*}{\small{$r$ [$\rm m$]} } & \multirow{3}{*}{\shortstack[l]{\footnotesize{Power allocation}\\ \footnotesize{~~~~~scheme}}} & \multicolumn{4}{c|}{\small{Hex deployment}} & \multicolumn{4}{c|}{\small{Hex+ deployment}}  & \multicolumn{4}{c|}{\small{ MCC deployment}} \\

\cline{3-14}
 & & \multicolumn{2}{c|}{\shortstack[l]{\footnotesize Data rate}} & \multicolumn{2}{c|}{\shortstack[l]{\footnotesize Illumination }} &
 \multicolumn{2}{c|}{\shortstack[l]{\footnotesize Data rate}} & \multicolumn{2}{c|}{\shortstack[l]{\footnotesize Illumination}} &
 \multicolumn{2}{c|}{\shortstack[l]{\footnotesize Data rate}} & \multicolumn{2}{c|}{\shortstack[l]{\footnotesize Illumination }}   \\
 
 \cline{3-14}
 & &{\shortstack[l]{{\hspace{-0.1cm}\scriptsize{$R_{\text{min}}$}}\\ \hspace{-0.1cm}\scriptsize [Mbps]}} & {\shortstack[l]{{\hspace{-0.15cm}\scriptsize{$R_{\text{mean}}$}}\\ \hspace{-0.1cm}\scriptsize [Mbps]}} &{\shortstack[l]{{\hspace{0.1cm}\scriptsize{$U_{\text{E}}$}}\\ \hspace{0.1cm}\scriptsize [1]}}& {\shortstack[l]{{\hspace{-0.05cm}\scriptsize{$E_{\text{avg}}$}}\\ \hspace{-0.05cm}\scriptsize [Lux]}}&
 {\shortstack[l]{{\hspace{-0.1cm}\scriptsize{$R_{\text{min}}$}}\\ \hspace{-0.1cm}\scriptsize [Mbps]}} & {\shortstack[l]{{\hspace{-0.15cm}\scriptsize{$R_{\text{mean}}$}}\\ \hspace{-0.1cm}\scriptsize [Mbps]}} &{\shortstack[l]{{\hspace{0.1cm}\scriptsize{$U_{\text{E}}$}}\\ \hspace{0.1cm}\scriptsize [1]}}& {\shortstack[l]{{\hspace{-0.05cm}\scriptsize{$E_{\text{avg}}$}}\\ \hspace{-0.05cm}\scriptsize [Lux]}}&
{\shortstack[l]{{\hspace{-0.1cm}\scriptsize{$R_{\text{min}}$}}\\ \hspace{-0.1cm}\scriptsize [Mbps]}} & {\shortstack[l]{{\hspace{-0.15cm}\scriptsize{$R_{\text{mean}}$}}\\ \hspace{-0.1cm}\scriptsize [Mbps]}} &{\shortstack[l]{{\hspace{0.1cm}\scriptsize{$U_{\text{E}}$}}\\ \hspace{0.1cm}\scriptsize [1]}}& {\shortstack[l]{{\hspace{-0.05cm}\scriptsize{$E_{\text{avg}}$}}\\ \hspace{-0.05cm}\scriptsize [Lux]}}   \\

\hline
 \multirow{2}{*}{~10} &\small{$P_{ \text{LED}}=40\rm W$} & \multirow{1}{*}{\cellcolor[HTML]{FFD1DC} 0} & \cellcolor[HTML]{FFFFED} 25.3  & \cellcolor[HTML]{AAF0D1} ~0 & \cellcolor[HTML]{A4DDED} 67 & \cellcolor[HTML]{FFD1DC} 7.7 & \cellcolor[HTML]{FFFFED} 24.3 & \cellcolor[HTML]{AAF0D1} 0.25 & \cellcolor[HTML]{A4DDED} 194 & \cellcolor[HTML]{FFD1DC} 9.1 & \cellcolor[HTML]{FFFFED} 29.0 & \cellcolor[HTML]{AAF0D1} 0.16 & \cellcolor[HTML]{A4DDED} 164\\

%%F08080
\cline{2-14}
 & \multirow{1}{*}{\small{$P_{\text{tot}}=1 \rm KW$}} & \multirow{1}{*}{\cellcolor[HTML]{FFD1DC} 0} & \cellcolor[HTML]{FFFFED} 46.1  & \cellcolor[HTML]{AAF0D1} ~0 & \cellcolor[HTML]{A4DDED} 559 & \cellcolor[HTML]{FFD1DC} 8.1 & \cellcolor[HTML]{FFFFED} 26.7 & \cellcolor[HTML]{AAF0D1} 0.25 & \cellcolor[HTML]{A4DDED} 346 & \cellcolor[HTML]{FFD1DC} 10.5 & \cellcolor[HTML]{FFFFED} 35.1 & \cellcolor[HTML]{AAF0D1} 0.16 & \cellcolor[HTML]{A4DDED} 512 \\%%

 \hline
\hline
\multirow{2}{*}{~~3} &\small{$P_{\text{LED}}=40\rm W$} & \multirow{1}{*}{\cellcolor[HTML]{FFD1DC} 0} & \cellcolor[HTML]{FFFFED} 50.9  & \cellcolor[HTML]{AAF0D1} 0.57 & \cellcolor[HTML]{A4DDED} 356 & \cellcolor[HTML]{FFD1DC} 7.3 & \cellcolor[HTML]{FFFFED} 25.8 & \cellcolor[HTML]{AAF0D1} 0.72 &\cellcolor[HTML]{A4DDED} 730 & \cellcolor[HTML]{FFD1DC} 10.1 & \cellcolor[HTML]{FFFFED} 42.4 &\cellcolor[HTML]{AAF0D1} 0.60 & \cellcolor[HTML]{A4DDED} 516\\

\cline{2-14}
 & \multirow{1}{*}{\small{$P_{\text{tot}}=1 \rm KW$}} & \multirow{1}{*}{\cellcolor[HTML]{FFD1DC} 0} & \cellcolor[HTML]{FFFFED} 54.7  & \cellcolor[HTML]{AAF0D1} 0.57 & \cellcolor[HTML]{A4DDED} 495 & \cellcolor[HTML]{FFD1DC} 7.3 & \cellcolor[HTML]{FFFFED} 24.0 & \cellcolor[HTML]{AAF0D1} 0.72 & \cellcolor[HTML]{A4DDED} 435 & \cellcolor[HTML]{FFD1DC} 10.1 & \cellcolor[HTML]{FFFFED} 42.8 &\cellcolor[HTML]{AAF0D1} 0.60 & \cellcolor[HTML]{A4DDED} 538 \\%%

 \hline
\hline
\multirow{2}{*}{~~2} &\small{$P_{\text{LED}}=40\rm W$} & \multirow{1}{*}{\cellcolor[HTML]{FFD1DC} 0} & \cellcolor[HTML]{FFFFED} 52.8  & \cellcolor[HTML]{AAF0D1} 0.61 & \cellcolor[HTML]{A4DDED} 694 & \cellcolor[HTML]{FFD1DC} 7.5& \cellcolor[HTML]{FFFFED} 33.7 & \cellcolor[HTML]{AAF0D1} 0.74 & \cellcolor[HTML]{A4DDED} 1211  & \cellcolor[HTML]{FFD1DC} 10.0 & \cellcolor[HTML]{FFFFED} 42.2 & \cellcolor[HTML]{AAF0D1} 0.66 & \cellcolor[HTML]{A4DDED} 989\\

\cline{2-14}
 & \multirow{1}{*}{\small{$P_{\text{tot}}=1 \rm KW$}} & \multirow{1}{*}{\cellcolor[HTML]{FFD1DC} 0} & \cellcolor[HTML]{FFFFED} 49.1  & \cellcolor[HTML]{AAF0D1} 0.61 & \cellcolor[HTML]{A4DDED} 495 & \cellcolor[HTML]{FFD1DC} 7.4 & \cellcolor[HTML]{FFFFED} 28.2 & \cellcolor[HTML]{AAF0D1} 0.74 & \cellcolor[HTML]{A4DDED} 452 & \cellcolor[HTML]{FFD1DC} 9.9 & \cellcolor[HTML]{FFFFED} 37.1 & \cellcolor[HTML]{AAF0D1} 0.66 &  \cellcolor[HTML]{A4DDED} 515 \\%%
\hline

\end{tabular}
%\caption{Caption}
%\label{label}
\vspace{-3mm}
\end{table*}

      % For better layout:
%\newpage
\pagestyle{empty}

\vspace{-2mm}
\section{Conclusions}
\label{sec:7}
\vspace{-0.5mm}

For the massive adoption of VLC technology, a process for the detailed planning of the locations of LED-based APs in a floor plan is necessary, to ensure that LoS connectivity can be established in optical wireless access and optical wireless backhaul.
%A procedure for the detailed planning of the locations that LED-based APs should take in a floor plan, to enable the LoS connectivity in the optical wireless access and optical wireless backhaul links, is required to promote the massive adoption of VLC technology.
To address this problem, a graph that models the common visibility areas of a floor plan partition was first proposed and, after that, two algorithms known as Maximal Clique Clustering~(MCC) and Connectivity Tree Construction~(CTC) were derived to minimize the number of VLC APs that are required for LoS condition in both access and backhaul links. The numbers of VLC APs identified by MCC and CTC were very similar, facilitating the ultra-dense deployment of a VLC network indoors using the same optical wireless technology for both access and backhauling. Aside from coverage, minimizing the number of APs using MCC also controls co-channel interference in the VLC network. Then, the mean data rate is maximized while the minimum data rate around the cell edges remains high enough. Additionally, MCC ensures that illumination comfort is achieved, which is linked to energy efficiency KPIs that smart buildings also aim to optimize.

\pagestyle{empty}

\vspace{0mm}  
\section*{APPENDIX: Properties of PV Graph modeling}
\vspace{0mm}
\subsection{Proof of Lemma \ref{lemma:non-empty-visibility}: non-empty visibility area }
As in Fig. \ref{fig:obstacles}(a), let $O$ and $r_0$ be the center point and the radius of the minimum enclosing circle of $p$, respectively, with vertices $X_1$, $X_2$, and $X_3$.  
  Since $O$ falls inside $p$ or on a side, no layout edge intersects the segments $OX_1$, $OX_2$, or $OX_3$ based on the definition of hyper triangulation. Thus, we only require to show that $OX_i\leq r$ holds for $i=1,2$ and $3$. Without loss of generality, we assume that $X_2X_3$ is the largest side of $p$. Thus, (i): $\|X_2X_3\|\leq R\leq \sqrt{3}r$. 
Furthermore, from $\widehat{X}_1\geq \pi/3$, we conclude that $2\pi/3\leq\widehat{O}\leq\pi$.  On the other hand, $X_2X_3$ is a chord of the minimum enclosing circle; thus, (ii): $\|X_2X_3\|=2r_0\sin{(\widehat{O}/2)}\geq\sqrt{3}r_0$. So from (i), (ii), and the definition of the minimum enclosing circle, we derive $OX_i\leq r_0\leq r$ for all $i$. Therefore, at least $O \in \mathcal{V}(p)$ and Lemma \ref{lemma:non-empty-visibility} follows.

\vspace{-2mm}
\subsection{Proof of Lemma \ref{lemma:Covering-visibility}: self-coverage of property}
%\vspace{-0.5mm}

%\section*{APPENDIX B}
%\textit{Sketch of proof for Lemma \ref{lemma:Covering-visibility}} :

In Fig.~\ref{fig:obstacles}(b), let us assume that $Q$ is an arbitrary point in $p$, $X_2X_3$ is the largest side and the line passing through $X_1$, and that $Q$ crosses $X_2X_3$ in $U$. Since no layout edge intersects $QX_1$, $QX_2$, or $QX_3$, it is enough to prove $\|QX_i\|\leq r$ for all $i$. Then, from \mbox{$\max (\widehat{X}_2,\widehat{X}_3) \leq \widehat{X}_1\leq \widehat{Q}$}, we conclude that $\max(\|QX_2\|,\|QX_3\|)\leq\|X_2X_3\|\leq r$. Furthermore, we have $\|QX_1\|\leq \|UX_1\|\leq \max(\|X_1X_3\|,\|X_1X_2\|)\leq\|X_2X_3\|\leq r$. As a result, $Q\in\mathcal{V}(p)$ and Lemma \ref{lemma:Covering-visibility} follows.

%\vspace{-2mm}
\subsection{Proof of Lemma \ref{Lemma:visibility_to_inside_triangle}: Full coverage property}
%\vspace{-0.5mm}
\begin{figure}[!t]
\centering
%%%%%%%%%%%%%%%%%figure%%%%%%%%%%%%%%%%%%%%%%
\hspace{0cm}\includegraphics[width=8cm]{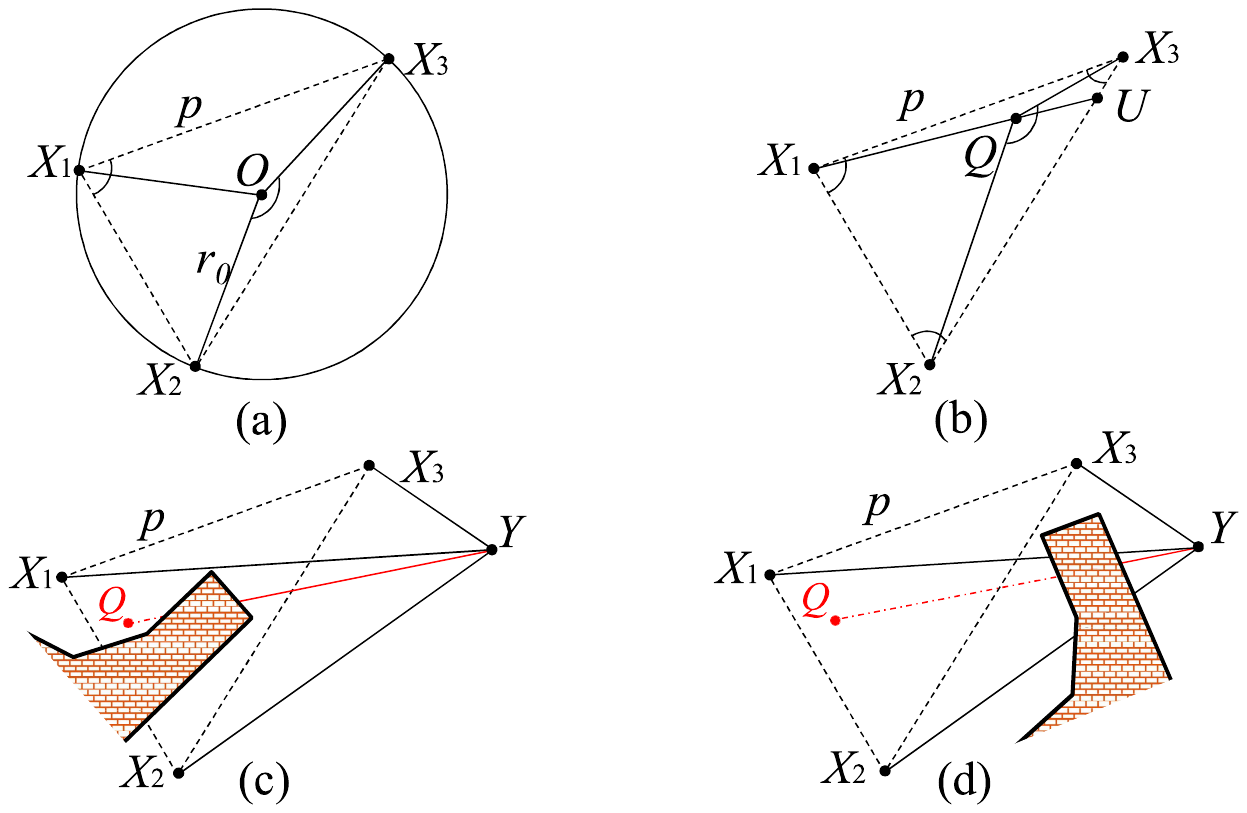}%{vtc_fig0.eps}
\vspace{-5.5cm}\caption{\   Illustrations used in the proofs of Lemmas related to properties of hyper triangulation process. }\label{fig:obstacles}
\vspace{-0.5 cm}
\end{figure}
Generally, segment $YQ$ lies inside the convex hull of the points $X_1$, $X_2$, $X_3$, and $Y$. Firstly, no layout edge crosses $YQ$. Otherwise, the layout edge crosses a boundary in the convex hull, i.e., one of the triangle sides as in Fig.~\ref{fig:obstacles}(c) or $YX_i$ for some $i$, as in Fig.~\ref{fig:obstacles}(d). The former contradicts the definition of hyper triangulation, and the latter contradicts $Y\in \mathcal{V}(p)$. Secondly, the farthermost point of a triangle to a point on the plane is one of the vertices. As a result, $\| YQ \|\leq \max( \| YX_1 \|,\| YX_2 \|, \| YX_3 \|) \leq r$, Thus, Lemma \ref{Lemma:visibility_to_inside_triangle} follows.

\vspace{-1mm}
\subsection{Proof of Theorem \ref{theorem:lowerbound-by-pvgraph}: Lower bound property}
%\vspace{-0.5mm}

Firstly, since there found a hidden set of $t_{\max}$ points in the layout, we have (i): $s_{\max}\geq t_{\max}$. Secondly, considering the largest hidden set of $s_{\max}$ points $U_1^{'}$, $U_2^{'}$, ..., $U_{s_{\max}}^{'}$ in the layout, Lemma \ref{Lemma:visibility_to_inside_triangle} implies that
 $\mathcal{V}(q_1^{'}), \mathcal{V}(q_2^{'}),..., \mathcal{V}(q_{s_{\max}}^{'})$ are all pairwise disjoint, where   $q_1^{'}$, $q_2^{'}$, \dots, $q_{s_{\max}}^{'}$ denote the set of  triangles that contain $U_1^{'}$, $U_2^{'}$, \dots, $U_{s_{\max}}^{'}$, respectively.
 Besides, any hyper triangulation from the space $\mathcal{HT}(R\leq 2r)$ ensures that any set of hidden points in the layout falls inside distinct triangles.
 Therefore,
$q_1^{'}$, $q_2^{'}$, \dots, $q_{s_{\max}}^{'}$ form an independent set of $s_{\max}$ nodes in the PV graph, thus (ii): $s_{\max}\leq t_{\max}$.
Finally, from (i) and (ii), we have $s_{\max}= t_{\max}$. \vspace{-1mm}
%%%%%%%%%%%%%%%%%%%%%%%%%%%%%%%%%%%%%%%%%%%%%%%%%%%%%%%%%%%%%%%%%%%%%%%%%%%%%%%%%%%%%%%%%%%%%%%%%%%%%%%%%%%%%%%%%%%%%%%%%%%%%%%%%%%%
%\vspace{1 cm}

%\begin{figure}[!t]
%\centering
%%%%%%%%%%%%%%%%%figure%%%%%%%%%%%%%%%%%%%%%%
%\hspace{0cm}\includegraphics[width=9cm]{fig1_AGP.eps}%{vtc_fig0.eps}
%\vspace{-0.2cm}\caption{\small   The properties of hyper triangulation. }\label{fig:obstacles}
%\vspace{-0.4 cm}
%\end{figure}

%\begin{figure}[!t]
%\centering
%%%%%%%%%%%%%%%%%figure%%%%%%%%%%%%%%%%%%%%%%
%\vspace{-0.3cm}\includegraphics[width=9cm]{fig2_AGP.eps}%{vtc_fig0.eps}
%\vspace{-0.2cm}\caption{\small   The %properties of hyper triangulation. %}\label{fig:obstacles}
%\vspace{-0.4 cm}
%\end{figure}

%\vspace{-2mm}
%\bibliographystyle{unsrt}
%\bibliographystyle{IEEEtran}
\bibliography{main.bib}
\bibliographystyle{ieeetr}

\end{document}